\documentclass[preprint,3p,times]{elsarticle}

\usepackage{amssymb}
\usepackage{amsmath}
\usepackage{amsthm}

\usepackage{algpseudocode}
\usepackage{algorithmicx}
\usepackage{bm}
\usepackage[shortlabels]{enumitem}
\usepackage{mathtools}
\mathtoolsset{showonlyrefs}

\newcommand{\FL}[1]{\mathrm{fl}\!\left( #1 \right)}

\newcommand{\FLd}[1]{\mathrm{fl}_{\bigtriangledown}\!\left( #1 \right)}
\newcommand{\FMA}[1]{\mathrm{fma}\!\left( #1 \right)}
\newtheorem{lem}{Lemma}
\newtheorem{thm}{Theorem}
\newtheorem{algorithm}{Algorithm}
\algrenewcommand\algorithmicindent{0.5em}
\usepackage[hyphens]{url}

\journal{Nuclear Physics B}

\begin{document}

\begin{frontmatter}

%% Title, authors and addresses

%% use the tnoteref command within \title for footnotes;
%% use the tnotetext command for theassociated footnote;
%% use the fnref command within \author or \affiliation for footnotes;
%% use the fntext command for theassociated footnote;
%% use the corref command within \author for corresponding author footnotes;
%% use the cortext command for theassociated footnote;
%% use the ead command for the email address,
%% and the form \ead[url] for the home page:
%% \title{Title\tnoteref{label1}}
%% \tnotetext[label1]{}
%% \author{Name\corref{cor1}\fnref{label2}}
%% \ead{email address}
%% \ead[url]{home page}
%% \fntext[label2]{}
%% \cortext[cor1]{}
%% \affiliation{organization={},
%%             addressline={},
%%             city={},
%%             postcode={},
%%             state={},
%%             country={}}
%% \fntext[label3]{}

\title{Error Analysis of Matrix Multiplication Emulation Using Ozaki-II Scheme}

%% use optional labels to link authors explicitly to addresses:
%% \author[label1,label2]{}
%% \affiliation[label1]{organization={},
%%             addressline={},
%%             city={},
%%             postcode={},
%%             state={},
%%             country={}}
%%
%% \affiliation[label2]{organization={},
%%             addressline={},
%%             city={},
%%             postcode={},
%%             state={},
%%             country={}}

\author[label1]{Yuki Uchino} %% Author name
\author[label2]{Katsuhisa Ozaki}
\author[label1]{Toshiyuki Imamura}

%% Author affiliation
\affiliation[label1]{%
	    organization={RIKEN Center for Computational Science},%Department and Organization
            addressline={7-1-26 Minatojima-minami-machi, Chuo-ku}, 
            city={Kobe},
            postcode={650-0047}, 
            state={Hyogo},
            country={Japan}}
\affiliation[label2]{%
	    organization={Department of Mathematical Sciences, Shibaura Institute of Technology},%Department and Organization
            addressline={307 Fukasaku, Minuma-ku}, 
            city={Saitama},
            postcode={337-8570}, 
            state={Saitama},
            country={Japan}}

%% Abstract
\begin{abstract}
%% Text of abstract
The Ozaki-II scheme is an emulation method that leverages the Chinese Remainder Theorem to compute high-precision matrix multiplication via a sequence of low-precision matrix multiplications.
In this scheme, the attainable numerical accuracy improves as the number of low-precision matrix multiplications increases.
Previous numerical studies have shown that single- and double-precision matrix multiplication using the Ozaki-II scheme achieves higher throughput than that of standard BLAS routines on modern AI hardware equipped with fast INT8 matrix multiply-accumulate units with INT8 inputs and INT32 accumulation.
However, the accuracy of the Ozaki-II scheme can degrade when the exponent distribution of the input matrices is wide, in which case a large number of low-precision matrix multiplications is required to obtain high-precision results.
In this paper, we present a rigorous deterministic error analysis of the Ozaki-II scheme.
The proposed analysis not only clarifies the accuracy behavior of the method but also enables the estimation of the number of low-precision matrix multiplications required to achieve a desired level of numerical accuracy.
\end{abstract}

%%Graphical abstract
% \begin{graphicalabstract}
%\includegraphics{grabs}
% \end{graphicalabstract}

%%Research highlights
% \begin{highlights}
% \item Research highlight 1
% \item Research highlight 2
% \end{highlights}

%% Keywords
\begin{keyword}
%% keywords here, in the form: keyword \sep keyword
error analysis \sep matrix multiplication \sep emulation \sep tensor core

%% PACS codes here, in the form: \PACS code \sep code

%% MSC codes here, in the form: \MSC code \sep code
%% or \MSC[2008] code \sep code (2000 is the default)
\MSC 65G50 \sep \MSC 68W40
\end{keyword}

\end{frontmatter}

%% Add \usepackage{lineno} before \begin{document} and uncomment 
%% following line to enable line numbers
%% \linenumbers

%% main text
%%

%====================
\section{Introduction}
\label{sec:Introduction}
%====================

%-----
\subsection{Introduction}
\label{subsec:Introduction}
%-----
Recent processors exhibit exceptional performance in low-precision arithmetic, enabling substantial acceleration of machine learning workloads.
There have been dramatic improvements in sub-32-bit precision operations, which are central to AI computation, with each new processor generation.
In contrast, the performance gains for single- and double-precision floating-point operations have relatively stagnated, directly contributing to performance stagnation for classical high-performance computing.
To address this challenge, several methods have been proposed for emulating high-precision matrix multiplication using multiple low-precision matrix multiplications~\cite{bf16x9,ootomo2022Recovering,Ootomo2023Quantum,mukunoki2020,ootomo2024dgemm,uchino2025Performance,mukunoki2025dgemmfp64arithmetic,ozaki-scheme2,uchino_ozaki2,uchino_ozaki2_complex}.
Among these approaches, we focus on the Ozaki-II scheme~\cite{ozaki-scheme2,uchino_ozaki2,uchino_ozaki2_complex}, which was previously proposed by the authors.
This scheme has the potential to achieve sufficient numerical accuracy with fewer low-precision matrix multiplications than those required for other emulation techniques.
Previous studies~\cite{uchino_ozaki2,uchino_ozaki2_complex} have demonstrated that the Ozaki-II scheme achieves higher performance than that of native FP32/FP64 matrix multiplication and other emulation methods on a variety of GPUs.

In this work, we provide a detailed error analysis of the emulation of single- and double-precision general matrix--matrix multiplication (SGEMM and DGEMM, respectively) based on the Ozaki-II scheme.
This analysis clarifies the numerical reliability of the scheme and serves as a foundation for future automatic tuning of accuracy-related parameters within the Ozaki-II scheme.

%-----
\subsection{Notation}
\label{subsec:Notations}
%-----
Assume that $A(i,:) \neq \mathbf{0}^T$ for $1 \le i \le m$ and $B(:,j) \neq \mathbf{0}$ for $1 \le j \le n$; that is, we exclude trivial cases in which the matrix product contains rows or columns that are identically zero.
For $b \in \mathbb{N}$, let $\mathbb{F}_{b}$ and $\mathbb{Z}_{b}$ denote the sets of $b$-bit binary floating-point numbers and $b$-bit signed integers, respectively.
Let $u_b$ denote the unit roundoff associated with $\mathbb{F}_{b}$.
For example, $\mathbb{F}_{32}$ and $\mathbb{F}_{64}$ correspond to the sets of single- and double-precision floating-point numbers (FP32 and FP64), respectively, while $\mathbb{Z}_{8}$ corresponds to the set of signed 8-bit integers (INT8).
The corresponding unit roundoffs are $u_{32} = 2^{-24}$ for FP32 and $u_{64} = 2^{-53}$ for FP64.
For any $x \in \mathbb{R}$, the function $\mathrm{round}(x)$ rounds $x$ off to the nearest integer, and the function $\mathrm{trunc}(x)$ returns the integer part of $x$.
For any $x \in \mathbb{Z}$ and $p \in \mathbb{N}$, we write 
\[
\bmod(x,p) := x - p \cdot \mathrm{round}(x/p).
\]
Then, $-\lfloor p/2 \rfloor \le \mathrm{mod}(x,p) \le \lfloor p/2 \rfloor$ holds.
We write $\mathrm{fl}(\cdot)$ and $\mathrm{fl}_{\bigtriangledown}(\cdot)$ to denote the numerical results; all operations inside the parentheses are computed by floating-point arithmetic in round-to-nearest-even and round-down modes, respectively.
For $x \in \mathbb{F}_{32}$, $\mathrm{log2f}(x)$ means the computed result of $\log_2 x$ in FP32.
In~\cite{uchino_ozaki2,uchino_ozaki2_complex,GEMMul8}, the fast built-in $\mathtt{\_\_log2f}(x)$ function is used.
In the same manner as in~\cite{uchino_ozaki2_complex}, we assume that
\begin{equation}\label{assumption:log2}
    |\mathrm{log2f}(x) - \log_2 x| \le 4u_{32}\log_2 x,
\end{equation}
according to the official documentation of CUDA~\cite{cuda} and HIP~\cite{hip}.
For $x \in \mathbb{R}$, let $\mathrm{single}(x)$, $\mathrm{single}_{\bigtriangledown}(x)$, and $\mathrm{single}_{\bigtriangleup}(x)$ denote the rounding of $x$ to FP32 values using round-to-nearest-even, round-down, and round-up modes, respectively.
Similarly, we define $\mathrm{double}(x)$ as the rounding of $x$ to an FP64 value using the round-to-nearest-even mode.
For $x,y,z \in \mathbb{F}_b$, the function $\mathrm{fma}(x,y,z)$ calculates $xy+z$ using a fused multiply-add (FMA) operation.
Let $\mathrm{ufp}(x)$ be defined as
\[
\mathrm{ufp}(x) := 
\begin{cases}
	0 & \text{if}\ x=0, \\
	2^{\lfloor \log_2 |x| \rfloor} & \text{otherwise}
\end{cases}
\]
for $x \in \mathbb{R}$. 
The notation applies elementwise to all entries unless otherwise stated.
For example, $X \le Y$, $Z := 2^X$, and $Z := \lceil X \rceil$ for $X,Y \in \mathbb{R}^{m \times n}$ mean $x_{ij} \le y_{ij}$, $z_{ij} := 2^{x_{ij}}$, and $z_{ij} := \lceil x_{ij} \rceil$ for all $(i,j)$ pairs, respectively.
Let $E$ be the all-ones matrix of appropriate size.
In this paper, mixed-precision matrix multiply-accumulate units with INT8 inputs, INT32 accumulation, and INT32 outputs are referred to as INT8 matrix engines.

%====================
\section{Related Work}
\label{sec:Related Work}
%====================

%====================
\subsection{Long-Multiplication-Based Methods}
\label{subsec:Long-Multiplication-Based Methods}
%====================
A number of emulation techniques for high-precision matrix multiplication are based on the long-multiplication decomposition of the input matrices.
The cuMpSGEMM method~\cite{ootomo2022Recovering,Ootomo2023Quantum} emulates SGEMM using either FP16 or TensorFloat-32 (TF32).
The BF16x9 algorithm~\cite{bf16x9} similarly emulates SGEMM using bfloat16 (BF16).
The Ozaki scheme (also referred to as the Ozaki-I scheme)~\cite{ozaki2012error,ozaki2013generalization} provides a general framework for arbitrary-precision matrix multiplication.
Several works have adapted this framework to emulate DGEMM using INT8, FP8, or FP16 arithmetic~\cite{mukunoki2020,ootomo2024dgemm,uchino2025Performance,mukunoki2025dgemmfp64arithmetic}.
These methods decompose the input matrices $A\in\mathbb{R}^{m\times k}$, $B\in\mathbb{R}^{k\times n}$ into unevaluated sums of low-precision matrices by slicing their significands:
\begin{equation}\label{eq:slice}
A \to A_1 + A_2 + \dots + A_d,\qquad
B \to B_1 + B_2 + \dots + B_d,
\end{equation}
where each slice is represented in a low-precision format through appropriate diagonal scaling.
An approximation of $AB$ can be obtained by computing these low-precision products.
In cuMpSGEMM, $A$ and $B$ are each decomposed into two FP16/TF32 matrices and a residual term, and FP32-accumulated matrix multiplications combined with an error correction step reproduce FP32 accuracy.
The BF16x9 algorithm decomposes each matrix into three BF16 matrices and achieves FP32 accuracy using FP32 accumulation.
In the Ozaki-I scheme, the matrices are sliced such that all slices share a common exponent range.
Consequently, the number of slices $d$ required to reach native FP64 accuracy depends on the exponent distribution of the input matrices.
A larger $d$ reduces the truncation error in~\eqref{eq:slice} and therefore yields higher accuracy.
Several error analyses of the Ozaki-I scheme have been published~\cite{uchino2025Performance, ozaki2012error, abdelfattah2025analysis, Abdelfattah2025Householder}.

%====================
\subsection{CRT-Based Methods}
\label{subsec:CRT-Based Methods}
%====================
Recently, we proposed an alternative framework for arbitrary-precision matrix multiplication emulation, namely the Ozaki-II scheme, based on the Chinese Remainder Theorem (CRT)~\cite{ozaki-scheme2}.
Building on this framework, emulation methods for SGEMM and DGEMM using INT8 matrix engines have been developed~\cite{uchino_ozaki2,uchino_ozaki2_complex,GEMMul8}.
Previous numerical studies~\cite{uchino_ozaki2,uchino_ozaki2_complex} demonstrated that the Ozaki-II scheme outperforms native general matrix--matrix multiplication routines and existing emulation approaches in both performance and power efficiency on a variety of GPUs.
The emulation methods based on the Ozaki-II scheme provide two computing modes: fast mode and accurate mode.
Although the fast mode has higher throughput than that of the accurate mode, its achievable accuracy strongly depends on the exponent distribution of the input matrices. 
In particular, when the exponent range is large, the fast mode fails to deliver accuracy comparable to that of native FP32 or FP64 computations.
In contrast, the accurate mode maintains stable accuracy even for inputs with a wide exponent distribution, making it suitable for general-purpose use.
For this reason, the present study focuses on the accurate mode.
The Ozaki-II scheme is a relatively recent technique and thus its numerical error characteristics have not yet been systematically analyzed.
In this work, we provide a detailed error analysis of SGEMM and DGEMM emulation based on the Ozaki-II scheme in the accurate mode.

The Ozaki-II scheme is based on the traditional method for computing high-precision integer matrix multiplication using the CRT.
The CRT provides a method for reconstructing the solutions of a sequence of congruence equations with pairwise coprime moduli from its residues, as shown in Theorem~\ref{thm:crt}.

\begin{thm}[Chinese Remainder Theorem]\label{thm:crt}
Let $x \in \mathbb{Z}$.
Suppose that $p_1,\dots,p_N \in \mathbb{N}_{\ge 2}$ are pairwise coprime integers and $\mathcal{P} := \prod_{1 \le \ell \le N}{p_\ell}$.
For $\ell=1,\dots,N$, define $q_\ell \in \mathbb{N}$ as modular multiplicative inverses of $\mathcal{P}/p_\ell$ (i.e., $\mathcal{P}/p_\ell \cdot q_\ell \equiv 1 \bmod p_\ell$).
Let $y_\ell\in \mathbb{Z}$ for $\ell=1,\dots,N$ be such that
\begin{equation}\label{eq:system}
\begin{cases}
    x \equiv y_1 \mod{p_1},\\
    \quad \vdots\\
    x \equiv y_N \mod{p_N}.
\end{cases}
\end{equation}
Then, it holds that
\begin{equation}\label{eq:xequiv}
  x \equiv \sum_{\ell=1}^{N} \frac{\mathcal{P}}{p_\ell} q_\ell y_\ell \mod{\mathcal{P}}.
\end{equation}
\end{thm}

Let $p_\ell \in \mathbb{N}_{\ge 2}$, $\mathcal{P} \in \mathbb{N}$, and $q_\ell \in \mathbb{N}$ be defined as in Theorem~\ref{thm:crt} for $\ell=1,\dots,N$.
In Theorem~\ref{thm:crt}, the solution $x$ can be represented as 
\[
x = \bmod\left( \sum_{\ell=1}^{N} \frac{\mathcal{P}}{p_\ell} q_\ell y_\ell, \mathcal{P} \right) + z\mathcal{P},\quad
z \in \mathbb{Z}.
\]
If $2|x| < \mathcal{P}$, $x$ becomes unique; that is, 
\[
x = \bmod\left( \sum_{\ell=1}^{N} \frac{\mathcal{P}}{p_\ell} q_\ell y_\ell, \mathcal{P} \right)
\]
holds.
For given matrices $A \in \mathbb{R}^{m \times k}$ and $B \in \mathbb{R}^{k \times n}$, the following is an overview of the Ozaki-II scheme:
\begin{enumerate}[Step~1),leftmargin=*,ref=\arabic*]
    \item\label{step1} Determine pairwise coprime integers $p_1,\dots,p_N \in \mathbb{N}_{\ge 2}$ and calculate $\mathcal{P} := \prod_{1 \le \ell \le N}{p_\ell} \in \mathbb{N}$ and modular multiplicative inverses $q_\ell \in \mathbb{N}$ of $\mathcal{P}/p_\ell$, which satisfy $q_\ell < p_\ell$.
    
    \item\label{step2} Apply diagonal scaling and truncation to convert $A$ and $B$ to 
    \begin{alignat}{2}
        A' &:= \mathrm{trunc}(\mathrm{diag}(2^\mu)\cdot A) \in \mathbb{Z}^{m \times k},&\quad \mu &\in \mathbb{Z}^{m},\label{eq:A'trunc}\\
        B' &:= \mathrm{trunc}(B\cdot \mathrm{diag}(2^\nu)) \in \mathbb{Z}^{k \times n},&\quad \nu &\in \mathbb{Z}^{n},\label{eq:B'trunc}
    \end{alignat}
    respectively, where $\mu$ and $\nu$ are chosen to satisfy
    \begin{equation}\label{eq:conditionofCRT}
        2|A'||B'| < \mathcal{P}E.
    \end{equation}
    
    \item\label{step3} Compute $C'' := A'B'$ via the CRT as
    \begin{align}
        A'_\ell &:= \mathrm{mod}(A',p_\ell),\quad B'_\ell := \mathrm{mod}(B',p_\ell),\label{eq:A'_ellB'_ell}\\
        C'_\ell &:= A'_\ell \cdot B'_\ell,\label{eq:C'_ell}\\
        C' &:= \sum_{\ell=1}^{N} \frac{\mathcal{P}}{p_\ell} q_\ell\cdot C'_\ell,\label{eq:C'}\\
        C'' &:=  \mathrm{mod}(C',\mathcal{P}).\label{eq:C''}
    \end{align}

    \item\label{step4} Inversely scale $C''$ as
    \[
        C := \mathrm{diag}(2^{-\mu}) \cdot C'' \cdot \mathrm{diag}(2^{-\nu}).
    \]
\end{enumerate}
The pairwise coprime integers $p_\ell \in \mathbb{N}_{\ge 2}$ are set as large as possible so that no rounding error occurs in the matrix multiplication $A'_\ell \cdot B'_\ell$ in~\eqref{eq:C'_ell}.
Condition~\eqref{eq:conditionofCRT} implies that the final reduction~\eqref{eq:C''} of the CRT produces a unique result.
The accumulation in~\eqref{eq:C'} should be performed in high precision.
Increasing $N$ enlarges $\mathcal{P}$ in~\eqref{eq:conditionofCRT}, which reduces the truncation error in~\eqref{eq:A'trunc} and~\eqref{eq:B'trunc}. 
Therefore, the accuracy of $C \approx AB$ depends on the number of moduli.

%====================
\section{Ozaki-II Scheme Using INT8 Matrix Engines}
\label{sec:Ozaki-II scheme using INT8 Matrix Engines}
%====================

We now introduce SGEMM and DGEMM emulation using INT8 matrix engines~\cite{uchino_ozaki2}.
We assume that $A \in \mathbb{F}_b^{m \times k}$ and $B \in \mathbb{F}_b^{k \times n}$ for $b \in \{32, 64\}$ and $k \le 2^{17}$.
Algorithm~\ref{alg:emu} shows the outline of the emulation.
This section describes the details of Algorithm~\ref{alg:emu}.

\begin{algorithm}[\cite{ozaki-scheme2,uchino_ozaki2,uchino_ozaki2_complex}]\label{alg:emu}
Let $A \in \mathbb{F}_b^{m \times k}$ and $B \in \mathbb{F}_b^{k \times n}$ for $b \in \{32, 64\}$ and $k \le 2^{17}$.
Let $p_1,\dots,p_N \in \mathbb{N}_{\ge 2}$ be pairwise coprime integers for $N \le 49$.
Assume that $p_\ell \le 256$ for $1 \le \ell \le N$.
The following algorithm computes $C \approx AB$ via the CRT using INT8 matrix engines.

\begin{algorithmic}[1]\setlength{\itemsep}{3pt}
\Function{$C := \bm{\mathit{OS\_II}}$}{$A,B,p$}

\State $[A',B',\mu,\nu] := \bm{\mathit{Scaling}}(A,B,p)$
\Comment{Algorithm~\ref{alg:scaling}, corresponding to Step~\ref{step2} in Section~\ref{subsec:CRT-Based Methods}}

\State $C'' := \bm{\mathit{CRT}}(A',B',p)$
\Comment{Algorithm~\ref{alg:crt}, corresponding to Step~\ref{step3} in Section~\ref{subsec:CRT-Based Methods}}

\State $C := \mathrm{fl}(\mathrm{diag}(2^{-\mu}) \cdot C'' \cdot \mathrm{diag}(2^{-\nu}))$
\Comment{$C \in \mathbb{F}_{b}^{m \times n}$, corresponding to Step~\ref{step4} in Section~\ref{subsec:CRT-Based Methods}}

\EndFunction
\end{algorithmic}
\end{algorithm}

%====================
\subsection{Determination of Constants}
\label{subsec:Determination of Constants}
%====================
The pairwise coprime integers $p_\ell$ are set as $p_\ell \le 256$, and thus $N \le 49$.
From the definition of $\mathrm{mod}$, $-\lfloor p_\ell/2 \rfloor \le \bmod(x,p_\ell) \le \lfloor p_\ell/2 \rfloor$ holds for $x \in \mathbb{Z}$.
When $p_\ell < 256$, $\bmod(x,p_\ell)$ can be held as an INT8 format value.
When $p_\ell = 256$, $\bmod(x,p_\ell)$ can be $128$, but this is not an issue because casting to INT8 wraps it around to $-128$ and $128 \equiv -128 \bmod 256$ holds.
In~\cite{GEMMul8}, $p$ is fixed as
\begin{equation}\label{p_list}
\begin{array}{rrrrrrrrrr@{}l}
p=(256, & 255, & 253, & 251, & 247, &
   241, & 239, & 233, & 229, & 227, \\
   223, & 217, & 211, & 199, & 197, &
   193, & 191, & 181, & 179, & 173, \\
   167, & 163, & 157, & 151, & 149, &
   139, & 137, & 131, & 127, & 113, \\
   109, & 107, & 103, & 101, &  97, &
    89, &  83, &  79, &  73, &  71, \\
    67, &  61, &  59, &  53, &  47, &
    43, &  41, &  37, &  29\phantom{,} &  &  )^T.
\end{array}
\end{equation}
In the implementation, not only the moduli $p$ but also constants related to $p$ are precomputed and stored as lookup tables.
This design choice reduces runtime overhead and contributes to the high efficiency of the emulation.
Note that the moduli $p$ that maximize $\mathcal{P}$ depend on $N$.
For instance, when $N \ge 6$, 255 results are excluded in a slightly larger $\mathcal{P}$; however, the difference is practically negligible for the accuracy of the final results in most cases.
In principle, separate lookup tables optimized for each value of $N$ can be constructed to further refine the implementation.
In this paper, we assume that $p$ is defined as in~\eqref{p_list}.

Both $\mathcal{P} \in \mathbb{N}$ and $\mathcal{P}/p_\ell \cdot q_\ell \in \mathbb{N}$ are prepared in high-precision format because $\mathrm{mod}(C',\mathcal{P}) = C' - \mathcal{P}\cdot \mathrm{round}(1/\mathcal{P} \cdot C')$ in \eqref{eq:C''} involves subtracting two nearly equal quantities, which introduces cancellation.
Therefore, \eqref{eq:C'} and \eqref{eq:C''} must be computed in high precision, since the output of \eqref{eq:C'} is directly used in \eqref{eq:C''}.
$\mathcal{P} \in \mathbb{N}$ is approximated by a double-double number $\mathcal{P}_1 + \mathcal{P}_2 \approx \mathcal{P}$ with $\mathcal{P}_1 = \mathrm{double}(\mathcal{P})$. 
We set $\mathcal{P}_2 := 0$ in SGEMM emulation and $\mathcal{P}_2 := \mathrm{double}(\mathcal{P} - \mathcal{P}_1)$ in DGEMM emulation.
$\mathcal{P}/p_\ell \cdot q_\ell$ is approximated using two double-precision floating-point numbers $s_{\ell 1},s_{\ell 2} \in \mathbb{F}_{64}$ in the form
\begin{equation}\label{def:s1s2}
\frac{\mathcal{P}}{p_\ell} \cdot q_\ell \approx s_{\ell 1} + s_{\ell 2}.
\end{equation}
In DGEMM emulation, $s_{\ell 1}$ retains only the upper $\beta_\ell$ bits of $\mathcal{P}/p_\ell \cdot q_\ell$ for
\[
\beta_\ell := 53 - \lceil \log_2 \rho \rceil + \left\lfloor \log_2 \frac{\mathcal{P}}{p_\ell}q_\ell \right\rfloor - \left\lfloor \log_2 \max_{1 \le h \le N} \frac{\mathcal{P}}{p_h}q_h \right\rfloor
\]
with $\rho := \sum_{\ell =1}^N \lfloor p_\ell/2 \rfloor$, while $s_{\ell 2}:= \mathrm{double}(\mathcal{P}/p_\ell \cdot q_\ell - s_{\ell 1})$, as shown in Figure~\ref{fig:img_s1s2}.
Note that in~\cite{uchino_ozaki2_complex}, $7 + \lceil \log_2 N \rceil$ is used instead of $\lceil \log_2 \rho \rceil$; however, using $\rho$ provides more accurate results from $\lceil \log_2 \rho \rceil \le 7 + \lceil \log_2 N \rceil$.
In SGEMM emulation, we set $s_{\ell 1} := \mathrm{double}(\mathcal{P}/p_\ell \cdot q_\ell)$ and $s_{\ell 2} := 0$.
\begin{figure}[htbp]
    \centering
    \includegraphics[width=.6\hsize]{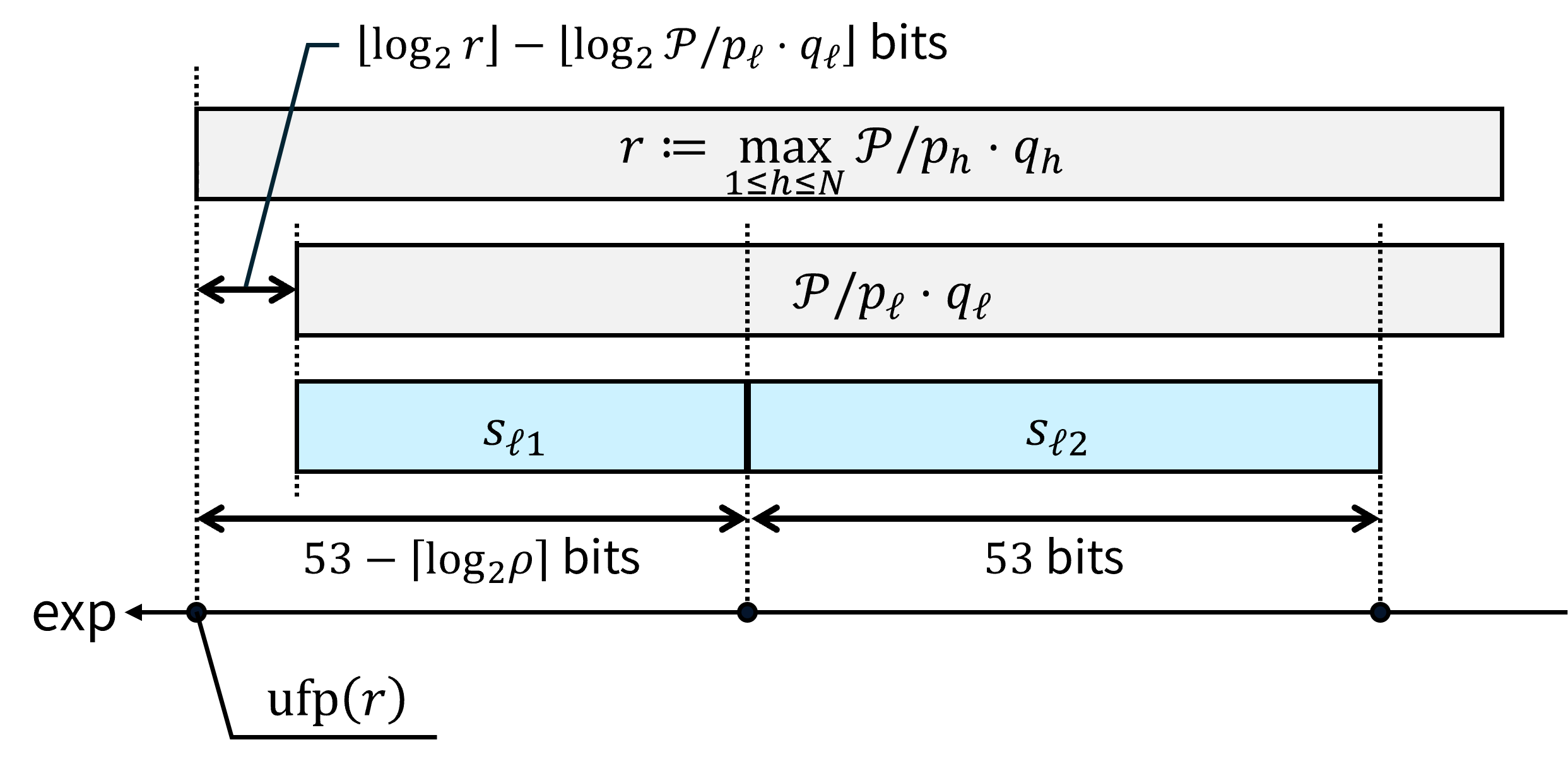}
    \caption{Diagram of $s_{\ell 1}$ and $s_{\ell 2}$.}
    \label{fig:img_s1s2}
\end{figure}

%====================
\subsection{Determination of Scaling Vectors}
\label{subsec:Determination of Scaling Vectors}
%====================
In Step~\ref{step2}, input floating-point matrices $A$ and $B$ are converted, using the scaling vectors $\mu \in \mathbb{Z}^m$ and $\nu \in \mathbb{Z}^n$, to integer matrices $A' \in \mathbb{F}_b^{m \times k} \cap \mathbb{Z}^{m \times k}$ and $B' \in \mathbb{F}_b^{k \times n} \cap \mathbb{Z}^{k \times n}$, respectively, such that condition~\eqref{eq:conditionofCRT} is satisfied.
From~\eqref{eq:A'trunc} and~\eqref{eq:B'trunc}, it holds that
\begin{equation}\label{eq:1}
|A'||B'| \le \mathrm{diag}(2^\mu)\cdot |A| |B| \cdot \mathrm{diag}(2^\nu).
\end{equation}
For $\mu' \in \mathbb{Z}^m$ and $\nu' \in \mathbb{Z}^n$, let 
\[
\bar{A} := \lceil \mathrm{diag}(2^{\mu'}) \cdot |A| \rceil \in \mathbb{Z}^{m \times k},\quad 
\bar{B} := \lceil |B| \cdot \mathrm{diag}(2^{\nu'}) \rceil \in \mathbb{Z}^{k \times n}.
\]
We define $\mu'$ and $\nu'$ as 
\begin{equation}\label{eq:mu'nu'}
\mu'_i := 5 - \left\lfloor \log_2 \max_{h}|a_{ih}| \right\rfloor,\quad
\nu'_j := 5 - \left\lfloor \log_2 \max_{h}|b_{hj}| \right\rfloor.
\end{equation}
Then, $0 \le \bar{a}_{ij}, \bar{b}_{ij} \le 2^6$ for all $(i,j)$ pairs.
Note that $\mu' \in \mathbb{Z}_{16}^m$ and $\nu' \in \mathbb{Z}_{16}^n$ because $\mu'_i, \nu'_j \in [-1018,1079] \cap \mathbb{Z}$ holds for all $(i,j)$ pairs from $\lfloor \log_2 |x| \rfloor \in [-1074,1023] \cap \mathbb{Z}$ for $x \in \mathbb{F}_{b}$ for $b \in \{32,64\}$.
Then,
\begin{equation}\label{eq:2}
|A| |B| \le \mathrm{diag}(2^{-\mu'})\cdot \bar{A} \bar{B} \cdot \mathrm{diag}(2^{-\nu'}),
\end{equation}
and $\bar{C} := \bar{A} \bar{B}$ is computed without error using INT8 matrix engines because we have
\begin{equation}\label{eq:2^29}
(\bar{A} \bar{B})_{ij} 
= \sum_{h=1}^k \bar{a}_{ih} \bar{b}_{hj} 
\le \sum_{h=1}^k (2^6)(2^6)
\le 2^{12}k
\le 2^{29}
< 2^{31}.
\end{equation}
From~\eqref{eq:1} and~\eqref{eq:2}, for $\bar{D} := \mathrm{single}_{\bigtriangleup}(\bar{C}) \in \mathbb{F}_{32}^{m \times n}$, we have
\begin{equation}\label{eq:3}
|A'||B'| 
\le \mathrm{diag}(2^{\mu-\mu'}) \cdot \bar{D} \cdot \mathrm{diag}(2^{\nu-\nu'}).
\end{equation}
To satisfy condition~\eqref{eq:conditionofCRT}, we can therefore determine $\mu \in \mathbb{Z}^m$ and $\nu \in \mathbb{Z}^n$ so that the following holds:
\begin{equation}\label{eq:|A'||B'|<=P/2}
    \mathrm{diag}(2^{\mu-\mu'}) \cdot \bar{D} \cdot \mathrm{diag}(2^{\nu-\nu'})
    < \frac{\mathcal{P}}{2}E.
\end{equation}
Algorithm~\ref{alg:scaling} shows the detailed computations.
The fact that the scaling vectors $\mu$ and $\nu$, defined as on lines~\ref{alg:emu:mu} and~\ref{alg:emu:nu}, respectively, satisfy condition~\eqref{eq:|A'||B'|<=P/2} is formalized as Lemma~\ref{lem0} in Section~\ref{sec:Theoretical Results}.
Its proof is provided in Section~\ref{sebsec:Proof of Lemma1}.
Note that $\mu \in \mathbb{Z}_{16}^m$ and $\nu \in \mathbb{Z}_{16}^n$ because $e_i,f_j < 31$ for all $(i,j)$ pairs and $\mathcal{P}' < 200$ for $N \le 49$.

\begin{algorithm}[\cite{uchino_ozaki2,uchino_ozaki2_complex}]\label{alg:scaling}
Let $A \in \mathbb{F}_b^{m \times k}$ and $B \in \mathbb{F}_b^{k \times n}$ for $b \in \{32, 64\}$.
Let $p \in \mathbb{N}_{\ge 2}^N$ with $49 \ge N \in \mathbb{N}_{\ge 2}$.
Assume that $k \le 2^{17}$.
The following algorithm produces $A' \in \mathbb{F}_b^{m \times k} \cap \mathbb{Z}^{m \times k}$ and $B' \in \mathbb{F}_b^{k \times n} \cap \mathbb{Z}^{k \times n}$ that satisfy~\eqref{eq:conditionofCRT}.

\begin{algorithmic}[1]\setlength{\itemsep}{3pt}
\Function{$[A',B',\mu,\nu] := \bm{\mathit{Scaling}}$}{$A,B,p$}
\State $\mathcal{P}' := \mathrm{single}_{\bigtriangledown}(\log_2(\prod_{\ell=1}^N p_\ell - 1)/2-0.5)$\label{alg:emu:P'}
\Comment{$\mathcal{P}' \in \mathbb{F}_{32}$}
\State $\mu'_i := 5 - \lfloor \log_2 \max_{h}|a_{ih}| \rfloor\ \forall i$\label{alg:emu:mu'}
\Comment{$\mu' \in \mathbb{Z}_{16}^{m}$}
\State $\bar{A} := \lceil\mathrm{diag}(2^{\mu'})\cdot |A|\rceil$\label{alg:emu:Abar}
\Comment{$\bar{A} \in \mathbb{Z}_8^{m \times k}$}
\State $\nu'_j := 5 - \left\lfloor \log_2 \max_{h}|b_{hj}| \right\rfloor\ \forall j$\label{alg:emu:nu'}
\Comment{$\nu' \in \mathbb{Z}_{16}^n$}
\State $\bar{B} := \lceil |B|\cdot \mathrm{diag}(2^{\nu'}) \rceil$\label{alg:emu:Bbar}
\Comment{$\bar{B} \in \mathbb{Z}_8^{k \times n}$}
\State $\bar{C} := \bar{A}\cdot \bar{B}$ using INT8 matrix engines
\Comment{$\bar{C} \in \mathbb{Z}_{32}^{m \times n}$}
\State $\bar{D} := \mathrm{single}_{\bigtriangleup}(\bar{C})$ using round-up mode\label{alg:emu:Dbar}
\Comment{$\bar{D} \in \mathbb{F}_{32}^{m \times n} \cap \mathbb{Z}^{m \times n}$}
\State $e_i := \mathrm{log2f}(\max_{h}\bar{d}_{ih})\ \forall i$\label{alg:scaling:log1}
\Comment{$e \in \mathbb{F}_{32}^m$}
\State $f_j := \mathrm{log2f}(\max_{h}\bar{d}_{hj})\ \forall j$\label{alg:scaling:log2}
\Comment{$f \in \mathbb{F}_{32}^n$}
\State $\mu_i := \mu'_i + \mathrm{fl}_{\bigtriangledown}(\left\lfloor \mathrm{fma}(\mathrm{single}_{\bigtriangledown}(-0.5/(1-4u_{32})), e_i,\mathcal{P}') \right\rfloor)\ \forall i$\label{alg:emu:mu}
\Comment{$\mu \in \mathbb{Z}_{16}^{m}$}
\State $\nu_j := \nu'_j + \mathrm{fl}_{\bigtriangledown}(\left\lfloor \mathrm{fma}(\mathrm{single}_{\bigtriangledown}(-0.5/(1-4u_{32})), f_j,\mathcal{P}') \right\rfloor)\ \forall j$\label{alg:emu:nu}
\Comment{$\nu \in \mathbb{Z}_{16}^n$}
\State $A' := \mathrm{trunc}( \mathrm{diag}(2^\mu)\cdot A)$\label{alg:emu:1}
\Comment{$A' \in \mathbb{F}_b^{m \times k} \cap \mathbb{Z}^{m \times k}$}
\State $B' := \mathrm{trunc}(B\cdot \mathrm{diag}(2^\nu))$\label{alg:emu:2}
\Comment{$B' \in \mathbb{F}_b^{k \times n} \cap \mathbb{Z}^{k \times n}$}
\EndFunction
\end{algorithmic}
\end{algorithm}

%====================
\subsection{Matrix Multiplication via INT8}
\label{subsec:Matrix Multiplication via INT8}
%====================
Let $A' \in \mathbb{F}_b^{m \times k} \cap \mathbb{Z}^{m \times k}$ and $B' \in \mathbb{F}_b^{k \times n} \cap \mathbb{Z}^{k \times n}$ be obtained using Algorithm~\ref{alg:scaling}.
Next, we convert $A'$ and $B'$ to INT8 matrices $A'_\ell$ and $B'_\ell$ as in~\eqref{eq:A'_ellB'_ell}, and let $C'_\ell := A'_\ell B'_\ell$ as in~\eqref{eq:C'_ell}.
When $p_\ell < 256$, $A'_\ell B'_\ell$ is computed without error because $\sum_{h=1}^k |(A'_\ell)_{ih}| |(B'_\ell)_{hj}| < 2^{31}$.
When $p_\ell = 256$, $(A'_\ell B'_\ell)_{ij}$ can be $2^{31}$; however, casting to INT32 wraps it around to $-2^{31}$ and $2^{31} \equiv -2^{31} \bmod 256$ holds.
After computing $C'_\ell$, we convert them to INT8 matrices $W_\ell := \bmod(C'_\ell,p_\ell) \in \mathbb{I}_{8}^{m \times n}$ to reduce the arithmetic precision of the accumulation in~\eqref{eq:C'} and thus mitigate the performance bottleneck; then, we accumulate them as $C' := \sum_{\ell=1}^{N} \mathcal{P}/p_\ell \cdot q_\ell\cdot W_\ell$ instead of~\eqref{eq:C'}.
Using $s_{\ell 1},s_{\ell 2} \in \mathbb{F}_{64}$ in Section~\ref{subsec:Determination of Constants}, $C' \approx C'^{(1)} + C'^{(2)}$ is computed as follows using FP64 arithmetic:
\begin{equation*}
    C'^{(1)} := \FL{\sum_{\ell=1}^{N}s_{\ell 1}W_\ell},\quad
    C'^{(2)} := \FL{\sum_{\ell=1}^{N}s_{\ell 2}W_\ell}.
\end{equation*}
In DGEMM emulation, no rounding error occurs in the computation of $C'^{(1)}$; the proof is provided in Section~\ref{sebsec:Proof of Lemma3}.
After the accumulation, we compute 
\begin{equation}\label{eq:Q}
    Q := \mathrm{round}\left(\FL{\mathcal{P}_{inv} \cdot C'^{(1)}}\right) = \mathrm{round}(C'/\mathcal{P}),
\end{equation}
where $\mathcal{P}_{inv} := \mathrm{double}(\mathcal{P}^{-1})$.
Then, for a double-double number $\mathcal{P}_1 + \mathcal{P}_2 \approx \mathcal{P}$ defined in Section~\ref{subsec:Determination of Constants}, we perform the final reduction of the CRT as
\[
    C'' := \FL{\FMA{-\mathcal{P}_2,Q,\FMA{-\mathcal{P}_1,Q,C'^{(1)}} + C'^{(2)}}} \approx C' - \mathcal{P}Q.
\]
Algorithm~\ref{alg:crt} shows the outline of computing $C'' \approx A'B'$ using INT8 matrix engines via the CRT.

\begin{algorithm}[\cite{uchino_ozaki2,uchino_ozaki2_complex}]\label{alg:crt}
Let $A' \in \mathbb{F}_b^{m \times k} \cap \mathbb{Z}^{m \times k}$ and $B' \in \mathbb{F}_b^{k \times n} \cap \mathbb{Z}^{k \times n}$ for $b \in \{32, 64\}$.
Let $p \in \mathbb{N}^N$ with $49 \ge N \in \mathbb{N}_{\ge 2}$.
Assume that $k \le 2^{17}$ and~\eqref{eq:conditionofCRT} hold.
The following algorithm computes $C'' \approx A'B'$ using INT8 matrix engines via the CRT.

\begin{algorithmic}[1]\setlength{\itemsep}{3pt}
\Function{$C'' := \bm{\mathit{CRT}}$}{$A',B',p$}
\State $\mathcal{P}_1 + \mathcal{P}_2 \approx \mathcal{P}$
\Comment{$\mathcal{P}_1, \mathcal{P}_2 \in \mathbb{F}_{64}$. $\mathcal{P}_2 = 0$ if $b = 32$}
\State $s_{\ell 1} + s_{\ell_2} \approx \mathcal{P}/p_\ell \cdot q_\ell$ for $1 \le \ell \le N$\label{alg:emu:s1s2}
\Comment{$s_{\ell 1},s_{\ell 2} \in \mathbb{F}_{64}$. $s_{i2} = 0$ if $b = 32$}
\State $\mathcal{P}_{inv} := \mathrm{double}(\mathcal{P}^{-1})$
\Comment{$\mathcal{P}_{inv} \in \mathbb{F}_{64}$}

\State $A'_\ell := \mathrm{mod}(A',p_i)$ for $1 \le \ell \le N$\label{alg:emu:3}
\Comment{$A'_\ell \in \mathbb{Z}_8^{m \times k}$}

\State$B'_\ell := \mathrm{mod}(B',p_i)$ for $1 \le \ell \le N$\label{alg:emu:4}
\Comment{$B'_\ell \in \mathbb{Z}_8^{k \times n}$}

\State $C'_\ell := A'_\ell B'_\ell$ for $1 \le \ell \le N$ using INT8 matrix engines
\Comment{$C'_\ell \in \mathbb{Z}_{32}^{m \times n}$}
\State $W_\ell := \bmod(C'_\ell,p_\ell)$ for $1 \le \ell \le N$\label{alg:emu:6}
\Comment{$W_\ell \in \mathbb{Z}_{8}^{m \times n}$}

\State $C'^{(1)} := \mathrm{fl}(\sum_{\ell=1}^N s_{\ell1} W_\ell)$
\Comment{$C'^{(1)} \in \mathbb{F}_{64}^{m \times n}$}
\State $C'^{(2)} := \mathrm{fl}(\sum_{\ell=1}^N s_{\ell2} W_\ell)$
\Comment{$C'^{(2)} \in \mathbb{F}_{64}^{m \times n}$}

\State $Q := \mathrm{round}(\mathrm{fl}(\mathcal{P}_{inv} \cdot C'^{(1)}))$
\Comment{$Q \in \mathbb{F}_{64}^{m \times n} \cap \mathbb{Z}^{m \times n}$}
\State $C'' := \mathrm{fl}(\mathrm{fma}(-Q, \mathcal{P}_2,\mathrm{fma}(-Q,\mathcal{P}_1,C'^{(1)}) + C'^{(2)}))$ 
\Comment{$C'' \in \mathbb{F}_{64}^{m \times n} \cap \mathbb{Z}^{m \times n}$}
\State $\bm{\mathit{if}}\ b=32,\ C'' := \mathrm{single}(C'')\ \bm{\mathit{end\ if}}$
\Comment{$C'' \in \mathbb{F}_{b}^{m \times n} \cap \mathbb{Z}^{m \times n}$}
\EndFunction
\end{algorithmic}
\end{algorithm}

%====================
\section{Theoretical Results}
\label{sec:Theoretical Results}
%====================

We provide the following theorem, which gives the error bound of the final result of DGEMM and SGEMM emulation.
\begin{thm}\label{thm:err}
    Let $p \in \mathbb{N}^N$ be pairwise coprime integers with $49 \ge N \in \mathbb{N}_{\ge 2}$.
    Assume that $k \le 2^{17}$ and $p_\ell \le 256$.
    Let $\mathcal{P} := \prod_{\ell=1}^N p_\ell$ and $\rho := \sum_{\ell=1}^N \lfloor p_\ell/2 \rfloor$.
    For $A \in \mathbb{F}_{b}^{m \times k}$ and $B \in \mathbb{F}_{b}^{k \times n}$ for $b \in \{32,64\}$, let $C_b$ be the result obtained by Algorithm~\ref{alg:emu}.
    Assume that $A(i,:) \neq \mathbf
    {0}^T$ for $1 \le i \le m$ and $B(:,j) \neq \mathbf{0}$ for $1 \le j \le n$.
    Let $\bar{C} \in \mathbb{Z}_{32}^{m \times n}$ be the quantity calculated in Algorithm~\ref{alg:scaling}. 
    Define $\alpha' \in \mathbb{Z}^{m}$ and $\beta' \in \mathbb{Z}^n$ as
    \[
        \alpha' := \alpha + \frac{1}{2}e',\quad
        \beta' := \beta + \frac{1}{2}f'
    \]
    for $\alpha,e' \in \mathbb{Z}^{m}$ and $\beta,f' \in \mathbb{Z}^n$ with 
    \begin{align*}
        \alpha_i := \lfloor \log_2 \max_{1 \le h \le k}|a_{ih}| \rfloor,\quad
        e'_i := \log_2\max_{1 \le h \le n}\bar{c}_{ih},\quad
        \beta_j := \lfloor \log_2 \max_{1 \le h \le k}|b_{hj}| \rfloor,\quad
        f'_j := \log_2\max_{1 \le h \le m}\bar{c}_{hj}.
    \end{align*}
    Let
    \begin{align*}
        R_{32} &:= (1 + u_{32})(N+2)u_{64}\rho\mathcal{P}E + u_{32}|A'B'|,\\
        R_{64} &:= (1 + 3u_{64})2^{1+\lceil \log_2 \rho \rceil}(N + 2)u_{64}^2\rho\mathcal{P}E + 3u_{64}|A'B'|.
    \end{align*}
    Then, for $v = (1,1,\dots,1)^T \in \mathbb{Z}^k$ and $t := 1/\sqrt{2^5(\mathcal{P}-1)}$, 
    \begin{align}
        |AB - C_{b}| 
        &\le t |A| v (2^{\beta'})^T
        + t 2^{\alpha'} v^T |B|
        + (kE + R_{b}) \circ t^2 2^{\alpha'} (2^{\beta'})^T,\label{thm:err-1}
    \end{align}
    where $\circ$ denotes elementwise multiplication.
\end{thm}

To facilitate the proof of Theorem~\ref{thm:err}, we introduce four lemmas below.
The proofs of Lemmas~\ref{lem1}, \ref{lem2}, \ref{lem3}, and~\ref{lem4}, together with the proof of Theorem~\ref{thm:err}, are presented in the next section.

\begin{lem}[Uniqueness of candidate of $A'B'$]\label{lem0}
    Let $p \in \mathbb{N}^N$ be pairwise coprime integers with $49 \ge N \in \mathbb{N}_{\ge 2}$.
    Let $A' \in \mathbb{F}_b^{m \times k} \cap \mathbb{Z}^{m \times k}$, $B' \in \mathbb{F}_b^{k \times n} \cap \mathbb{Z}^{k \times n}$ for $b \in \{32, 64\}$ be the quantities calculated in Algorithm~\ref{alg:scaling}.
    Assume that $k \le 2^{17}$ and $p_\ell \le 256$.
    Let $\mathcal{P} := \prod_{\ell=1}^N p_\ell$.
    Then, $|A'||B'| < \mathcal{P}/2 \cdot E$ holds.
\end{lem}

\begin{lem}[Truncation error]\label{lem1}
    Let $p \in \mathbb{N}^N$ be pairwise coprime integers with $49 \ge N \in \mathbb{N}_{\ge 2}$.
    Let $A' \in \mathbb{F}_b^{m \times k} \cap \mathbb{Z}^{m \times k}$, $B' \in \mathbb{F}_b^{k \times n} \cap \mathbb{Z}^{k \times n}$, $\bar{C} \in \mathbb{Z}_{32}^{m \times n}$, $\mu \in \mathbb{Z}_{16}^m$, and $\nu \in \mathbb{Z}_{16}^n$ be the quantities calculated in Algorithm~\ref{alg:scaling} for $b \in \{32, 64\}$.
    Assume that $k \le 2^{17}$ and $p_\ell \le 256$.
    Let $\mathcal{P} := \prod_{\ell=1}^N p_\ell$.
    Define $\alpha' \in \mathbb{Z}^{m}$ and $\beta' \in \mathbb{Z}^n$ as
    \[
        \alpha' := \alpha + \frac{1}{2}e',\quad
        \beta' := \beta + \frac{1}{2}f'
    \]
    for $\alpha,e' \in \mathbb{Z}^{m}$ and $\beta,f' \in \mathbb{Z}^n$ with 
    \begin{align*}
        \alpha_i := \lfloor \log_2 \max_{1 \le h \le k}|a_{ih}| \rfloor,\quad
        e'_i := \log_2\max_{1 \le h \le n}\bar{c}_{ih},\quad
        \beta_j := \lfloor \log_2 \max_{1 \le h \le k}|b_{hj}| \rfloor,\quad
        f'_j := \log_2\max_{1 \le h \le m}\bar{c}_{hj}.
    \end{align*}
    Let $v = (1,1,\dots,1)^T \in \mathbb{Z}^k$ and $t := 1/\sqrt{2^5(\mathcal{P}-1)}$.
    Then, 
    \begin{align}
        \mu_i &\ge -\alpha'_i + \frac{1}{2}\left( \log_2(\mathcal{P}-1) + 5 \right),\label{eq:mu_i_upper}\\
        \nu_i &\ge -\beta'_i + \frac{1}{2}\left( \log_2(\mathcal{P}-1) + 5 \right),\label{eq:nu_i_upper}
    \end{align}
    and
    \begin{align}
        \left| AB - \mathrm{diag}(2^{-\mu}) \cdot A'B' \cdot \mathrm{diag}(2^{-\nu}) \right| 
        \le t |A| v (2^{\beta'})^T
            + t 2^{\alpha'} v^T |B|
            + kt^2 2^{\alpha'} (2^{\beta'})^T.\label{eq:AB_upper}
    \end{align}
    hold.
\end{lem}

\begin{lem}[Accumulation error]\label{lem2}
    Let $A' \in \mathbb{F}_b^{m \times k} \cap \mathbb{Z}^{m \times k}$ and $B' \in \mathbb{F}_b^{k \times n} \cap \mathbb{Z}^{k \times n}$ for $b \in \{32, 64\}$.
    Let $p \in \mathbb{N}^N$ be pairwise coprime integers with $49 \ge N \in \mathbb{N}_{\ge 2}$.
    Let $\mathcal{P} := \prod_{\ell=1}^N p_\ell$.
    Let $q_\ell \in \mathbb{N}$ be the modular multiplicative inverses of $\mathcal{P}/p_{\ell}$.
    Assume that $k \le 2^{17}$, $p_\ell \le 256$, and $q_\ell < p_\ell$.
    Let $C'^{(1)},C'^{(2)} \in \mathbb{F}_{64}^{m \times n}$ and $W_\ell \in \mathbb{Z}_8^{m \times n}$ be the quantities calculated in Algorithm~\ref{alg:crt}.
    Let $\rho := \sum_{\ell =1}^N \lfloor p_\ell/2 \rfloor$.
    Then, for $b = 64$, 
    \[
        \left| C'^{(1)} + C'^{(2)} - \sum_{\ell = 1}^N \frac{\mathcal{P}}{p_\ell} q_\ell W_\ell \right| 
        \le 2^{1+\lceil \log_2 \rho \rceil}(N + 1 + Nu_{64})u_{64}^2 \rho \mathcal{P}E
    \]
    holds. For $b = 32$,
    \[
        \left| C'^{(1)} - \sum_{\ell = 1}^N \frac{\mathcal{P}}{p_\ell} q_\ell W_\ell \right| 
        \le (N+1)u_{64} \rho \mathcal{P}E
    \]
    holds.
\end{lem}

\begin{lem}[Validity of $Q$]\label{lem3}
    Let $Q \in \mathbb{F}_{64}^{m \times n} \cap \mathbb{Z}^{m \times n}$ and $W_\ell \in \mathbb{Z}_8^{m \times n}$ be the quantities calculated in Algorithm~\ref{alg:crt}.
    Let $p \in \mathbb{N}^N$ be pairwise coprime integers with $49 \ge N \in \mathbb{N}_{\ge 2}$.
    Assume that $p_\ell \le 256$.
    Suppose that $q_\ell \in \mathbb{N}$ is the modular multiplicative inverse of $\prod_{h=1}^N p_h/p_\ell$.
    Let $Q_{\mathit{exact}} := \mathrm{round}(\sum_{\ell=1}^N q_\ell/p_\ell \cdot W_\ell)$.
    Then, $Q = Q_{\mathit{exact}}$ holds.
\end{lem}

\begin{lem}[Error in final reduction]\label{lem4}
    Let $A' \in \mathbb{F}_b^{m \times k} \cap \mathbb{Z}^{m \times k}$, $B' \in \mathbb{F}_b^{k \times n} \cap \mathbb{Z}^{k \times n}$ for $b \in \{32, 64\}$.
    Let $p \in \mathbb{N}^N$ be pairwise coprime integers with $49 \ge N \in \mathbb{N}_{\ge 2}$.
    Assume that $k \le 2^{17}$ and $p_\ell \le 256$.
    Let $\rho := \sum_{\ell =1}^N \lfloor p_\ell/2 \rfloor$ and $\mathcal{P} := \prod_{\ell=1}^N p_\ell$.
    Let $C''_b \in \mathbb{F}_{b}^{m \times n} \cap \mathbb{Z}^{m \times n}$ be the quantity calculated in Algorithm~\ref{alg:crt}.
    Let
    \begin{align*}
        R_{32} &:= (1 + u_{32})(N+2)u_{64}\rho\mathcal{P}E + u_{32}|A'B'|,\\
        R_{64} &:= (1 + 3u_{64})2^{1+\lceil \log_2 \rho \rceil}(N + 2)u_{64}^2\rho\mathcal{P}E + 3u_{64}|A'B'|.
    \end{align*}
    Then, 
    \[
        |A'B' - C''_{b}| \le R_{b}
    \]
    holds.
\end{lem}

We have $|A'||B'| < \mathcal{P}/2 \cdot E$ from Lemma~\ref{lem0}; thus, based on \eqref{thm:err-1} in Theorem~\ref{thm:err}, we obtain
\begin{align}
        |AB - C_{b}| 
        &\le t |A| v (2^{\beta'})^T
        + t 2^{\alpha'} v^T |B|
        + (kE + R_{b}) \circ t^2 2^{\alpha'} (2^{\beta'})^T \\
        &\le t |A| v (2^{\beta'})^T + t 2^{\alpha'} v^T |B| + (k + r_b)t^2 2^{\alpha'} (2^{\beta'})^T, \label{thm:err-2}
\end{align}
where 
\begin{align*}
    r_{32} &:= (1 + u_{32})(N+2)u_{64}\rho\mathcal{P} + \frac{1}{2}u_{32}\mathcal{P},\\
    r_{64} &:=(1 + 3u_{64})2^{1+\lceil \log_2 \rho \rceil}(N + 2)u_{64}^2\rho\mathcal{P} + \frac{3}{2}u_{64}\mathcal{P}.
\end{align*}

%====================
\section{Proofs}
\label{sec:Proofs}
%====================

%====================
\subsection{Properties}
%====================
We give the following lemmas for use in proofs.
\begin{lem}[Definition of floating-point numbers]\label{lem:fl}
Let $a \in \mathbb{R}$ be an integral multiple of the minimum positive floating-point number.
If $|a|$ is less than or equal to the maximum floating-point number, it holds that
\begin{equation*}
    a \in \mathbb{F}_b \iff \exists k \in \mathbb{Z}\ \mathrm{s.t.}\ |a| \le 2^k,\ a \in u_b 2^k\mathbb{Z}. 
\end{equation*}
\end{lem}

\begin{lem}[\cite{rump2008accurate,jeannerod2018relative}]\label{lem:basic-err}
For $a \in \mathbb{R}$, let $\hat{a} \in \mathbb{F}_{b}$ be a nearest floating-point number of $a$.
Assume that $\hat{a} \neq 0$.
Suppose that $u'_b := u_b/(1+u_b)$.
Then, 
\begin{equation*}
    |a - \hat{a}| \le u_b \cdot \mathrm{ufp}(a) \le u_b \cdot \mathrm{ufp}(\hat{a})
\end{equation*}
and
\begin{equation*}
    |a - \hat{a}| \le u'_b |a|
\end{equation*}
hold.
\end{lem}

\begin{lem}[\cite{Jeannerod2013improved}]\label{lem:prod-err}
For $x,y \in \mathbb{F}^n_b$, barring underflow and overflow, 
\[
|x^Ty - \FL{x^Ty}| \le nu_b|x^T||y|.
\]
\end{lem}

%====================
\subsection{Proof of Lemma~\ref{lem0}}
\label{sebsec:Proof of Lemma1}
%====================
\begin{proof}
From $\bar{D} := \mathrm{single}_{\bigtriangleup}(\bar{C}) \ge \bar{C} = \bar{A} \bar{B}$ and \eqref{assumption:log2}, we have
\begin{align}
\mathrm{log2f}(\bar{D})
\ge (1-4u_{32})\log_2\bar{D}
\ge (1-4u_{32})\log_2\bar{C}.\label{eq:log2Cbar}
\end{align}
In addition, $\mathrm{single}_{\bigtriangledown}(-0.5/(1-4u_{32})) \le -0.5(1 + u_{32})/(1-4u_{32})$ holds.
Using this and~\eqref{eq:log2Cbar} yields
\begin{align*}
    \mathrm{single}_{\bigtriangledown}\left(\frac{-0.5}{1-4u_{32}}\right) \cdot \mathrm{log2f}(\bar{D})
    \le -0.5(1 + u_{32})\log_2\bar{C}.
\end{align*}
Thus, for $e_i := \mathrm{log2f}(\max_h\bar{d}_{ih})$ and $e'_i := \log_2 \max_h\bar{c}_{ih}$, we have
\begin{align*}
    \left\lfloor \FLd{\FMA{\mathrm{single}_{\bigtriangledown}\left(\frac{-0.5}{1-4u_{32}}\right),e_i,\mathcal{P'}}} \right\rfloor
    \le \left\lfloor \mathcal{P}' - 0.5(1 + u_{32})e'_i \right\rfloor.
\end{align*}
For all $\mathcal{P}' = \mathrm{single}_{\bigtriangledown}(\log_2(\prod_{\ell=1}^N p_\ell - 1)/2-0.5)$ computed from the moduli $p_\ell$ defined in~\eqref{p_list} and for all $e'_i$ satisfying $0 \le e'_i \le 29$, we verified by exhaustive evaluation that 
\[
\lfloor \mathcal{P}' - 0.5(1 + u_{32})e'_i \rfloor \le \mathcal{P}' - 0.5e'_i - 2^{-21},
\]
where the constant $2^{-21}$ on the right-hand side was obtained empirically from this exhaustive computation.
Therefore, we derive
\begin{align*}
    \left\lfloor \FLd{\FMA{\mathrm{single}_{\bigtriangledown}\left(\frac{-0.5}{1-4u_{32}}\right),e_i,\mathcal{P'}}} \right\rfloor
    \le \mathcal{P}' - 0.5e'_i - 2^{-21}
    \le \frac{1}{2}\left( \log_2(\mathcal{P}-1) - 1 - e'_i - 2^{-20} \right).
\end{align*}
Similarly, for $f_j := \mathrm{log2f}(\max_h\bar{d}_{hj})$ and $f'_j := \log_2 \max_h\bar{c}_{hj}$, we obtain
\begin{align*}
    \left\lfloor \FLd{\FMA{\mathrm{single}_{\bigtriangledown}\left(\frac{-0.5}{1-4u_{32}}\right),f_j,\mathcal{P'}}} \right\rfloor
    \le \frac{1}{2}\left( \log_2(\mathcal{P}-1) - 1 - f'_j - 2^{-20} \right).
\end{align*}
For $\alpha_i := \lfloor \log_2 \max_{1 \le h \le k}|a_{ih}| \rfloor$ and $\beta_j := \lfloor \log_2 \max_{1 \le h \le k}|b_{hj}| \rfloor$, we obtain
\begin{align*}
    \mu_i 
    &\le \mu'_i + \frac{1}{2}\left( \log_2(\mathcal{P}-1) - 1 - e'_i \right)
    = 5-\alpha_i + \frac{1}{2}\left( \log_2(\mathcal{P}-1) - 1 - e'_i - 2^{-20} \right),\\
    \nu_j 
    &\le \nu'_j + \frac{1}{2}\left( \log_2(\mathcal{P}-1) - 1 - f'_j \right)
    = 5-\beta_j + \frac{1}{2}\left( \log_2(\mathcal{P}-1) - 1 - f'_j - 2^{-20} \right)
\end{align*}
from~\eqref{eq:mu'nu'} and Algorithm~\ref{alg:scaling}.
Therefore, we obtain
\begin{align}
(|A'||B'|)_{ij}
&\le 2^{\mu_i} \cdot (|A||B|)_{ij} \cdot 2^{\nu_j}\\
&\le 2^{5-\alpha_i + (\log_2(\mathcal{P}-1)-1-e'_i-2^{-20})/2} \cdot (|A||B|)_{ij} \cdot 2^{5-\beta_j + (\log_2(\mathcal{P}-1)-1-f'_j-2^{-20})/2}\\
&= (|A||B|)_{ij} \cdot 2^{5-\alpha_i+5-\beta_j + \log_2(\mathcal{P}-1)-1-2^{-20} - (e'_i + f'_j)/2}\\
&\le (|A||B|)_{ij} \cdot 2^{5-\alpha_i+5-\beta_j} \cdot (\mathcal{P}-1) \cdot 2^{-1-2^{-20}} \cdot \bar{c}_{ij}^{-1}\\
&\le (|A||B|)_{ij} \cdot 2^{5-\alpha_i+5-\beta_j} \cdot (\mathcal{P}-1) \cdot 2^{-1-2^{-20}} \cdot (2^{5-\alpha_i}(|A||B|)_{ij}2^{5-\beta_j})^{-1}\\
&= (\mathcal{P}-1) 2^{-1-2^{-20}}\label{eq:|A||B|upper-rigorous}\\
&< \mathcal{P}/2.\label{eq:|A||B|upper}
\end{align}

\end{proof}

%====================
\subsection{Proof of Lemma~\ref{lem1}}
%====================
\begin{proof}
From~\eqref{assumption:log2}, for $O \le \Delta_1 < 2u_{32}E$ and $|\Delta_2| \le 4u_{32}E$, we derive 
\begin{align*}
    \mathrm{log2f}(\bar{d}_{ij}) 
    = ( 1 + (\Delta_2)_{ij} )\log_2((1 + (\Delta_1)_{ij})\bar{c}_{ij})
    = ( 1 + (\Delta_2)_{ij} )(\log_2(1 + (\Delta_1)_{ij}) + \log_2\bar{c}_{ij}).
\end{align*}
This implies that
\begin{align*}
    \mathrm{log2f}(\bar{D})
    \le (1 + 4u_{32})(3u_{32}E + \log_2\bar{C}).
\end{align*}
From this and $\mathrm{single}_{\bigtriangledown}(-0.5/(1-4u_{32})) \ge -0.5(1 + 6u_{32})$,
\begin{align*}
    \mathrm{single}_{\bigtriangledown}\left(\frac{-0.5}{1-4u_{32}}\right) \cdot \mathrm{log2f}(\bar{D})
    &\ge -0.5(1 + 6u_{32})(1 + 4u_{32})(3u_{32}E + \log_2\bar{C})\\
    &= -0.5\log_2\bar{C} - (5u_{32} + 12u_{32}^2)\log_2\bar{C} - 1.5u_{32}(1 + 10u_{32} + 24u_{32}^2)E
\end{align*}
We have $\log_2\bar{C} \le 29$ from~\eqref{eq:2^29}; thus, 
\begin{align*}
    \mathrm{single}_{\bigtriangledown}\left(\frac{-0.5}{1-4u_{32}}\right) \cdot \mathrm{log2f}(\bar{D})
    \ge -0.5\log_2\bar{C} - 147u_{32}E.
\end{align*}
Thus, for $e_i := \mathrm{log2f}(\max_h\bar{d}_{ih})$ and $e'_i := \log_2 \max_h\bar{c}_{ih}$, we have
\begin{align*}
    \left\lfloor \FLd{\FMA{\mathrm{single}_{\bigtriangledown}\left(\frac{-0.5}{1-4u_{32}}\right),e_i,\mathcal{P'}}} \right\rfloor
    &\ge \left\lfloor (1-2u_{32})\left( \mathcal{P}' - 0.5e'_i - 147u_{32} \right) \right\rfloor\\
    &\ge \left( \frac{\log_2(\mathcal{P}-1)}{2} - 0.5 \right) - \frac{e'_i}{2} - 2\\
    &\ge \frac{1}{2}\left( \log_2(\mathcal{P}-1) - e'_i - 5 \right).
\end{align*}
Similarly, for $f_j := \mathrm{log2f}(\max_h\bar{d}_{hj})$ and $f'_j := \log_2 \max_h\bar{c}_{hj}$, we obtain
\begin{align*}
    \left\lfloor \FLd{\FMA{\mathrm{single}_{\bigtriangledown}\left(\frac{-0.5}{1-4u_{32}}\right),f_j,\mathcal{P'}}} \right\rfloor
    \ge \frac{1}{2}\left( \log_2(\mathcal{P}-1) - f'_j - 5 \right).
\end{align*}
For $\alpha_i := \lfloor \log_2 \max_{1 \le h \le k}|a_{ih}| \rfloor$ and $\beta_j := \lfloor \log_2 \max_{1 \le h \le k}|b_{hj}| \rfloor$, we obtain
\begin{align*}
    \mu_i 
    &\ge \mu'_i + \frac{1}{2}\left( \log_2(\mathcal{P}-1) - e'_i - 5 \right)
    = -\alpha_i + \frac{1}{2}\left( \log_2(\mathcal{P}-1) - e'_i + 5 \right),\\
    \nu_j 
    &\ge \nu'_j + \frac{1}{2}\left( \log_2(\mathcal{P}-1) - f'_j - 5 \right)
    = -\beta_j + \frac{1}{2}\left( \log_2(\mathcal{P}-1) - f'_j + 5 \right)
\end{align*}
from~\eqref{eq:mu'nu'} and Algorithm~\ref{alg:scaling}.

From $A' := \mathrm{trunc}(\mathrm{diag}(2^{\mu}) \cdot A)$ and $B' := \mathrm{trunc}(B \cdot \mathrm{diag}(2^{\nu}))$, 
the quantities $\Delta_A := A' - \mathrm{diag}(2^{\mu}) \cdot A$ and $\Delta_B := B' - B \cdot \mathrm{diag}(2^{\nu})$ satisfy $|(\Delta_A)_{ij}| < 1$ and $|(\Delta_B)_{ij}| < 1$.
For $t := 1/\sqrt{2^{5}(\mathcal{P}-1)}$, $\alpha'_i := \alpha_i + e'_i/2$, and $\beta'_j := \beta_j + f'_j/2$, we have $2^{-\mu_i} \le t2^{\alpha'_i}$ and $2^{-\nu_j} \le t2^{\beta'_j}$.
Hence, for $v = (1,1,\dots,1)^T \in \mathbb{Z}^k$, we obtain
\begin{align*}
    |\mathrm{diag}(2^{-\mu}) \Delta_A| &< \mathrm{diag}(t2^{\alpha'}) E = t 2^{\alpha'} \cdot v^T,\\
    |\Delta_B \mathrm{diag}(2^{-\nu})| &< E\cdot \mathrm{diag}(t2^{\beta'}) = v \cdot (t 2^{\beta'})^T.
\end{align*}
Therefore, 
\begin{align*}
    | AB - \mathrm{diag}(2^{-\mu}) \cdot A'B' \cdot \mathrm{diag}(2^{-\nu}) |
    &= |AB - (A + \mathrm{diag}(2^{-\mu}) \Delta_A)(B + \Delta_B \mathrm{diag}(2^{-\nu}))|\\
    &\le |A| |\Delta_B \mathrm{diag}(2^{-\nu})| + |\mathrm{diag}(2^{-\mu}) \Delta_A| |B| + |\mathrm{diag}(2^{-\mu}) \Delta_A| |\Delta_B \mathrm{diag}(2^{-\nu})|\\
    &\le t |A| v (2^{\beta'})^T
    + t 2^{\alpha'} v^T |B|
    + kt^2 2^{\alpha'} (2^{\beta'})^T.
\end{align*}

\end{proof}

%====================
\subsection{Proof of Lemma~\ref{lem2}}
\label{sebsec:Proof of Lemma3}
%====================
\begin{proof}
First, we give the proof for $b = 64$.
For $r := \max_{1 \le \ell \le N} \mathcal{P}/p_\ell \cdot q_\ell$, $0 < r < \mathcal{P}$ holds.
From the definition of $s_{\ell 1}$, for $\rho = \sum_{\ell =1}^N \lfloor p_\ell/2 \rfloor$, we have
\begin{equation}\label{eq:s_i1}
    0 \le s_{\ell 1} < 2\mathrm{ufp}(r),\quad
    s_{\ell 1} \in 2^{\lceil \log_2 \rho \rceil} u_{64} \cdot 2\mathrm{ufp}(r)\mathbb{Z}.
\end{equation}
With $W_\ell \in \mathbb{Z}_8^{m \times n}$, we derive 
\[
    |s_{\ell 1} W_\ell| < \left\lfloor \frac{p_\ell}{2} \right\rfloor \cdot 2\mathrm{ufp}(r) E,\quad
    s_{\ell 1} W_\ell \in 2^{\lceil \log_2 \rho \rceil} u_{64} \cdot 2\mathrm{ufp}(r)\mathbb{Z}^{m \times n}.
\]
Therefore, $(s_{\ell 1} W_\ell)_{ij}$ complies with the definition of a double-precision floating-point number as Lemma~\ref{lem:fl}; thus, $\FL{s_{\ell 1} W_\ell} = s_{\ell 1} W_\ell$ holds.
In addition,
\[
    \sum_{\ell=1}^N |s_{\ell 1} W_\ell| \le \rho \cdot 2\mathrm{ufp}(r)E,\quad
    \sum_{\ell =1}^N s_{\ell 1} W_\ell \in 2^{\lceil \log_2 \rho \rceil} u_{64} \cdot 2\mathrm{ufp}(r)\mathbb{Z}^{m \times n}.
\]
Thus, Lemma~\ref{lem:fl} implies that $C'^{(1)} = \mathrm{fl}(\sum_{\ell =1}^N s_{\ell 1} W_\ell) = \sum_{\ell =1}^N s_{\ell 1} W_\ell$.

From Lemma~\ref{lem:basic-err}, $s_{\ell 2} = \mathrm{double}(\mathcal{P}/p_\ell \cdot q_\ell - s_{\ell 1}) = \mathcal{P}/p_\ell \cdot q_\ell - s_{\ell 1} + \delta_{s_{\ell 2}}$, where
\begin{equation}\label{eq:delta_s_{i2}}
    |\delta_{s_{\ell 2}}| 
    \le u'_{64} \cdot \left(\frac{\mathcal{P}}{p_\ell} q_\ell - s_{\ell 1}\right).
\end{equation}
For $\Delta_{C'^{(2)}} := \mathrm{fl}(\sum_{\ell =1}^N s_{\ell 2} W_\ell) - \sum_{\ell =1}^N s_{\ell2} W_\ell$, we have
\begin{align}
C'^{(1)} + C'^{(2)} 
= \sum_{\ell=1}^N s_{\ell1} W_\ell + \sum_{\ell=1}^N s_{\ell2} W_\ell + \Delta_{C'^{(2)}}
= \sum_{\ell=1}^N \frac{\mathcal{P}}{p_\ell} q_\ell W_\ell + \sum_{\ell=1}^N \delta_{s_{\ell2}} W_\ell + \Delta_{C'^{(2)}}\label{eq:C'^{(1)} + C'^{(2)}}.
\end{align}
From~\eqref{eq:delta_s_{i2}},
\begin{equation}\label{eq:deltaW}
    \left| \sum_{\ell=1}^N \delta_{s_{\ell2}} W_\ell \right|
    \le u'_{64}\sum_{\ell =1}^N \left(\frac{\mathcal{P}}{p_\ell} q_\ell - s_{\ell 1}\right) |W_\ell|.
\end{equation}
Lemma~\ref{lem:prod-err} implies that
\begin{align}
    \left|\Delta_{C'^{(2)}}\right| 
    \le Nu_{64} \sum_{\ell =1}^N |s_{\ell2} W_\ell|
    \le Nu_{64}(1+u'_{64})\sum_{\ell =1}^N \left(\frac{\mathcal{P}}{p_\ell} q_\ell - s_{\ell 1}\right) |W_\ell|.
    \label{eq:Delta_C'^{(2)}}
\end{align}
From~\eqref{eq:C'^{(1)} + C'^{(2)}}, \eqref{eq:deltaW}, and \eqref{eq:Delta_C'^{(2)}}, for $\Delta_{C'} := \sum_{\ell=1}^N \mathcal{P}/p_\ell \cdot q_\ell\cdot W_\ell - (C'^{(1)} + C'^{(2)})$,
\begin{align}
    \left| \Delta_{C'} \right|
    \le (Nu_{64} (1+u'_{64})+u'_{64})\sum_{\ell =1}^N \left(\frac{\mathcal{P}}{p_\ell} q_\ell - s_{\ell 1}\right) |W_\ell|.
\end{align}
Now, we have $\mathcal{P}/p_\ell \cdot q_\ell - s_{\ell 1} < 2^{1+\lceil \log_2 \rho \rceil}u_{64}\mathcal{P}$ from~\eqref{eq:s_i1}, $|W_\ell| \le \lfloor p_\ell/2 \rfloor E$ from the definition of the $\bmod$ operation, and $\rho = \sum_{\ell =1}^N \lfloor p_\ell/2 \rfloor$ from the definition of $\rho$; thus, we obtain
\begin{align}
    \left| \Delta_{C'} \right|
    \le (Nu_{64} (1+u'_{64})+u'_{64}) \cdot 2^{1+\lceil \log_2 \rho \rceil}u_{64}\mathcal{P} \cdot \sum_{\ell =1}^N \left\lfloor \frac{p_\ell}{2} \right\rfloor E
    \le 2^{1+\lceil \log_2 \rho \rceil}(N + 1 + Nu_{64})u_{64}^2 \rho \mathcal{P}E.
    \label{eq:err_C'}
\end{align}

Next, we give the proof for $b = 32$.
From $s_{\ell 1} = \mathrm{double}(\mathcal{P}/p_\ell \cdot q_\ell)$ and Lemma~\ref{lem:basic-err}, 
\begin{equation}\label{eq:s1_32}
    \left| \frac{\mathcal{P}}{p_\ell} q_\ell - s_{\ell 1} \right| 
    \le u'_{64}\frac{\mathcal{P}}{p_\ell} q_\ell
    \le u'_{64}\mathcal{P}.
\end{equation}
Lemma~\ref{lem:prod-err} and $s_{\ell 1} \le \mathcal{P}$ imply that
\begin{equation}\label{eq:errC'1}
    \left| C'^{(1)} - \sum_{\ell=1}^N s_{\ell 1} W_\ell \right|
    \le Nu_{64} \sum_{\ell=1}^N s_{\ell 1} |W_\ell|
    \le Nu_{64} \rho \mathcal{P} E.
\end{equation}
Thus, \eqref{eq:s1_32} and \eqref{eq:errC'1} yield
\begin{align}
    \left| C'^{(1)} - \sum_{\ell=1}^N \frac{\mathcal{P}}{p_\ell} q_\ell W_\ell \right|
   &\le \left| C'^{(1)} - \sum_{\ell=1}^N s_{\ell 1} W_\ell \right|
      + \left| \sum_{\ell=1}^N s_{\ell 1} W_\ell - \sum_{\ell=1}^N \frac{\mathcal{P}}{p_\ell} q_\ell W_\ell \right|\\
   &\le Nu_{64} \rho \mathcal{P} E + u'_{64} \rho \mathcal{P} E\\
   &\le (N+1)u_{64} \rho \mathcal{P}E.\label{eq:err_C'_32}
\end{align}
\end{proof}

%====================
\subsection{Proof of Lemma~\ref{lem3}}
%====================
\begin{proof}
For $C'_{\mathit{exact}} := \sum_{\ell=1}^N \mathcal{P}/p_\ell \cdot q_\ell \cdot W_\ell$ and $Q_{\mathit{exact}} := \mathrm{round}(\mathcal{P}^{-1} C'_{\mathit{exact}})$,
\begin{equation}\label{eq:A'B'}
A'B' 
= \bmod(C'_{\mathit{exact}},\mathcal{P})
= C'_{\mathit{exact}} - \mathcal{P}Q_{\mathit{exact}} \in \mathbb{Z}^{m \times n}
\end{equation}
holds.
Define $\Delta_{Q_1} \in \mathbb{R}^{m \times n}$ as
\begin{equation}\label{def:deltaQ}
\Delta_{Q_1} := Q_{\mathit{exact}} - \mathcal{P}^{-1}C'_{\mathit{exact}}.
\end{equation}
Then, we have
\begin{equation}\label{eq:deltaQ}
    \left|\Delta_{Q_1}\right| 
    \le \frac{\mathcal{P}-1}{2^{1+2^{-20}}\mathcal{P}}E
    < \frac{\mathcal{P}-1}{2(1 + 2^{-21})\mathcal{P}}E
\end{equation}
because $|\mathcal{P} \Delta_{Q_1}| = |A'B'| \le |A'||B'| \le (\mathcal{P}-1) 2^{-1-2^{-20}}$ from \eqref{eq:|A||B|upper-rigorous}.

For $b=64$, from $s_{\ell 1} < \mathcal{P}$,
\begin{equation}\label{eq:C'1upper64}
|C'^{(1)}| 
= \left| \sum_{\ell =1}^N s_{\ell 1} W_\ell \right|
\le \rho \mathcal{P}E.
\end{equation}
For $b = 32$, from Lemma~\ref{lem:prod-err}, we have
\begin{equation}
|C'^{(1)}| \le (1+Nu_{64})\left| \sum_{\ell =1}^N s_{\ell 1} W_\ell \right|.
\end{equation}
From Lemma~\ref{lem:basic-err} and $q_\ell \le p_\ell-1$,
\begin{align*}
(1+Nu_{64})s_{\ell 1} 
\le (1 + Nu_{64})(1 + u'_{64})\frac{\mathcal{P}}{p_{\ell}}q_{\ell}
\le (1 + Nu_{64})(1 + u'_{64})\left(\mathcal{P} - \frac{\mathcal{P}}{p_{\ell}} \right) \le \mathcal{P}.
\end{align*}
Therefore, we also have
\begin{equation}\label{eq:C'1upper32}
|C'^{(1)}| \le \rho\mathcal{P}E.
\end{equation}
Thus, from \eqref{eq:C'1upper64}, \eqref{eq:C'1upper32}, and Lemma~\ref{lem:basic-err}, we obtain
\begin{align}
    \left|\mathcal{P}_{inv}C'^{(1)} - \FL{\mathcal{P}_{inv}C'^{(1)}}\right|
    \le u'_{64}\mathcal{P}_{inv}|C'^{(1)}|
    \le u'_{64}\mathcal{P}_{inv} \cdot \rho\mathcal{P}E
    \le u'_{64}\mathcal{P}^{-1}(1 + u'_{64}) \cdot \rho\mathcal{P}E
    = \rho u'_{64}(1+u'_{64})E.\label{eq:Qproof1}
\end{align}
Moreover, \eqref{eq:s_i1} implies $|\mathcal{P}/p_\ell \cdot q_\ell - s_{\ell 1}| \le 2^{1 + \lceil \log_2 \rho \rceil}u_{64}\mathcal{P}$ for $b=64$ and \eqref{eq:s1_32} indicates $|\mathcal{P}/p_\ell \cdot q_\ell - s_{\ell 1}| \le u_{64} \mathcal{P}$ for $b = 32$.
From these and~\eqref{eq:C'1upper32}, we derive
\begin{align}
    \left|\mathcal{P}^{-1} C'_{\mathit{exact}} - \mathcal{P}_{inv} C'^{(1)}\right|
    &\le \mathcal{P}^{-1} \left|C'_{\mathit{exact}} - C'^{(1)}\right| + \mathcal{P}^{-1}u'_{64} |C'^{(1)}|\\
    &\le \mathcal{P}^{-1} \left|\sum_{i=1}^N \left(\frac{\mathcal{P}}{p_i}q_i - s_{i1}\right)W_i\right| + \rho u'^2_{64}E\\
    &\le \rho \cdot 2^{1 + \lceil \log_2 \rho \rceil} u_{64} E + \rho u'^2_{64}E\\
    &= \rho (2^{1 + \lceil \log_2 \rho \rceil}u_{64} + u'^2_{64})E.\label{eq:Qproof2}
\end{align}
From \eqref{eq:Qproof1} and \eqref{eq:Qproof2}, for $\Delta_{Q_2} :=\mathcal{P}^{-1} C'_{\mathit{exact}} - \FL{\mathcal{P}_{inv} C'^{(1)}}$, we have
\begin{align}
    |\Delta_{Q_2}|&\le \left|\mathcal{P}^{-1}C'_{\mathit{exact}} - \mathcal{P}_{inv}C'^{(1)}\right| +  \left|\mathcal{P}_{inv}C'^{(1)} - \FL{\mathcal{P}_{inv}C'^{(1)}}\right|\\
    &\le \rho (2^{1 + \lceil \log_2 \rho \rceil}u_{64} + u'^2_{64})E + \rho u'_{64}(1+u'_{64})E\\
    &= \rho (2^{1 + \lceil \log_2 \rho \rceil}u_{64} + u'^2_{64} + 2u'^2_{64})E \\
    &\le \rho (2^{1 + \lceil \log_2 \rho \rceil} + 2)u_{64}E.
    \label{eq:errQ}
\end{align}
From $N \le 49$ and $\rho = \sum_{\ell=1}^N \lfloor p_\ell/2 \rfloor \le 2^7N \le 2^7 49 < 2^{13}$, we have
\[
    |\Delta_{Q_2}| 
    \le \rho (2^{1 + \lceil \log_2 \rho \rceil} + 2)u_{64}E
    \le 2^{7}49 (2^{1 + 13} + 2)u_{64}E
    < 2^{27}u_{64}E.
\]
Therefore, \eqref{eq:deltaQ} and this imply that
\begin{align*}
\left|Q_{\mathit{exact}} - \FL{\mathcal{P}_{inv}C'^{(1)}}\right| 
&\le \left| Q_{exacxt} - \mathcal{P}^{-1}C'_{\mathit{exact}} \right| + \left| \mathcal{P}^{-1} C'_{\mathit{exact}} - \FL{\mathcal{P}_{inv}C'^{(1)}} \right|\\
&\le \Delta_{Q_1} + \Delta_{Q_2}\\
&\le \frac{\mathcal{P}-1}{2(1 + 2^{-21})\mathcal{P}}E + 2^{27}u_{64}E\\
&= \frac{\mathcal{P}-1 + 2^{27}u_{64} \cdot 2(1 + 2^{-21})\mathcal{P}}{2(1 + 2^{-21})\mathcal{P}}E\\
&< \frac{(1 + 2^{-24})\mathcal{P}-1}{2(1 + 2^{-21})\mathcal{P}}E\\
&< \frac{E}{2}.
\end{align*}
Thus, we obtain $Q = \mathrm{round}(\mathrm{fl}(\mathcal{P}_{inv}C'^{(1)})) = Q_{\mathit{exact}}$, as shown in Figure~\ref{fig:round}.
\begin{figure}[htbp]
    \centering
    \includegraphics[width=.6\hsize]{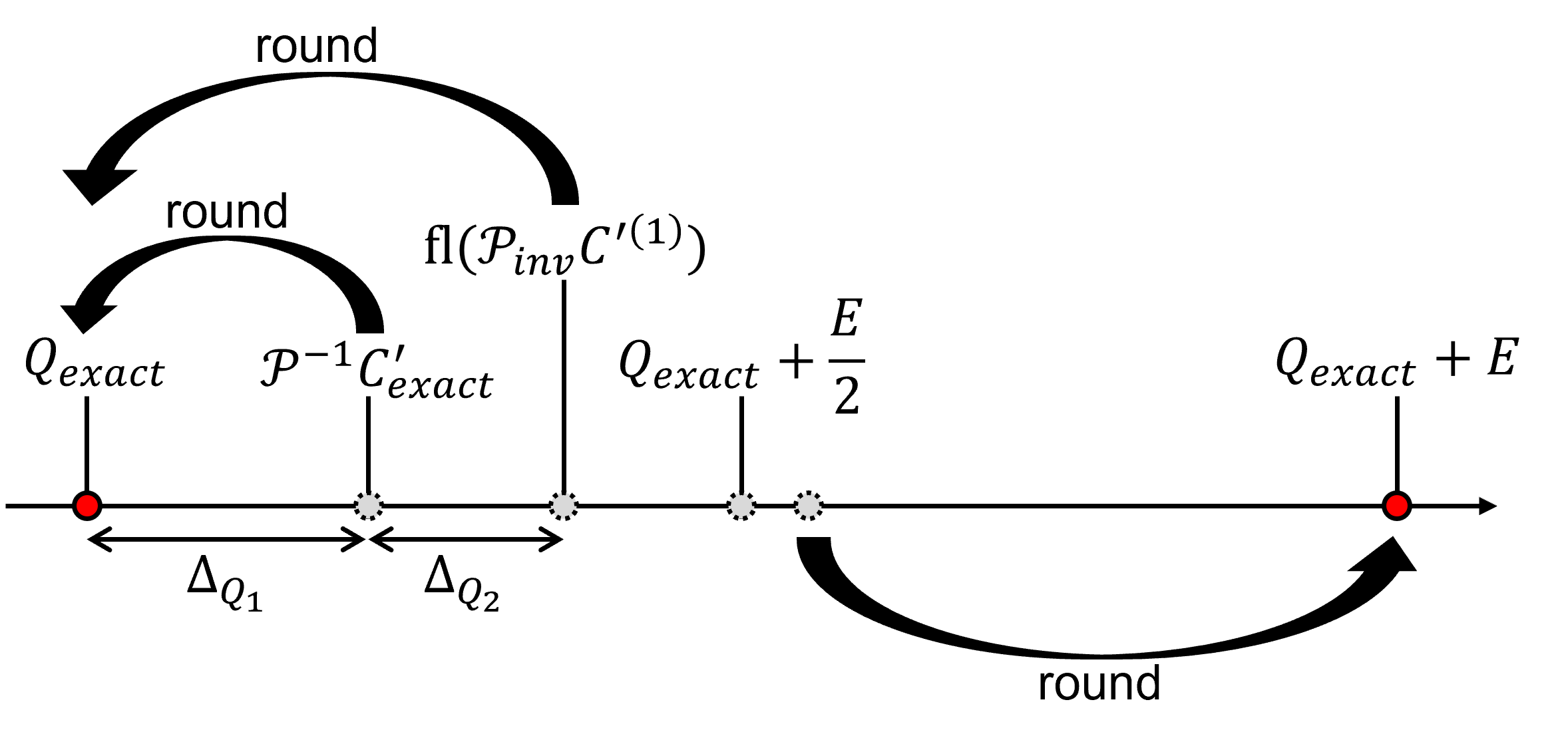}
    \caption{Diagram of $Q$, $\Delta_{Q_1}$, and $\Delta_{Q_2}$.}
    \label{fig:round}
\end{figure}
\end{proof}

%====================
\subsection{Proof of Lemma~\ref{lem4}}
%====================
\begin{proof}
Let $C'_{\mathit{exact}} := \sum_{\ell=1}^N \mathcal{P}/p_\ell \cdot q_\ell \cdot W_\ell$, $D_{64} := C'^{(1)} + C'^{(2)} - (\mathcal{P}_1 + \mathcal{P}_2)Q$, and $D_{32} := C'^{(1)}_{ij} - \mathcal{P}_1 Q_{ij}$.
We analyze the error between $A'B'$ and $C''_{b}$ based on
\begin{equation}\label{eq:errA'B'}
    \left| A'B' - C''_{b} \right| \le \left| A'B' - D_{b} \right| + \left|D_{b} - C''_{b}\right|.
\end{equation}
First, we give the proof for $b = 64$.
From the definition of a double-double number, 
\begin{equation}\label{eq:errP}
    |(\mathcal{P}_1 + \mathcal{P}_2) - \mathcal{P}| \le u_{64}^2\frac{\mathcal{P}}{2}.
\end{equation}
In addition, we have
\begin{align}
    |Q| 
    \le \left| \mathcal{P}^{-1}C'_{\mathit{exact}} \right| + \frac{E}{2}
    = \sum_{i=1}^N \frac{q_i}{p_i} |W_i| + \frac{E}{2}
    \le \left( \rho + \frac{1}{2}\right) E.\label{eq:upperboundQ}
\end{align}
Thus, from~\eqref{eq:err_C'}, \eqref{eq:A'B'}, \eqref{eq:errP}, and \eqref{eq:upperboundQ}, 
\begin{align}
\left| A'B' - D_{64} \right|
&\le \left| C'_{\mathit{exact}} - (C'^{(1)} + C'^{(2)}) \right| 
+ \left| \mathcal{P} - (\mathcal{P}_1 + \mathcal{P}_2) \right| \left| Q \right|\\
&\le 2^{1+\lceil \log_2 \rho \rceil}(N + 1 + Nu_{64})u_{64}^2\rho\mathcal{P}E + \frac{u_{64}^2}{2}\mathcal{P}\left( \rho + \frac{1}{2}\right)E\\
&\le 2^{1+\lceil \log_2 \rho \rceil}(N + 1 + Nu_{64} + 2^{-2-\lceil \log_2 \rho \rceil})u_{64}^2\rho\mathcal{P}E + \frac{1}{4}u_{64}^2\mathcal{P}E\\
&\le 2^{1+\lceil \log_2 \rho \rceil}(N + 2)u_{64}^2\rho\mathcal{P}E.
\label{eq:A'B'-D64}
\end{align}
From Lemma~\ref{lem:basic-err}, for $1 - u'_{64} \le (\Delta_1)_{ij},(\Delta_2)_{ij},(\Delta_3)_{ij} \le 1 + u'_{64}$, we have
\begin{align}
    (C''_{64})_{ij}
    &= \FL{\FMA{-Q_{ij}, \mathcal{P}_2,\FMA{-Q_{ij},\mathcal{P}_1,C'^{(1)}_{ij}} + C'^{(2)}_{ij}}}\\
    &= (\Delta_1)_{ij}\left( \FL{\FMA{-Q_{ij},\mathcal{P}_1,C'^{(1)}_{ij}} + C'^{(2)}_{ij}} - \mathcal{P}_2 Q_{ij} \right)\\
    &= (\Delta_1)_{ij}\left( (\Delta_2)_{ij}\left( \FL{\FMA{-Q_{ij},\mathcal{P}_1,C'^{(1)}_{ij}}} + C'^{(2)}_{ij} \right) - \mathcal{P}_2 Q_{ij} \right)\\
    &= (\Delta_1)_{ij}\left( (\Delta_2)_{ij}\left( (\Delta_3)_{ij}\left( C'^{(1)}_{ij} - \mathcal{P}_1 Q_{ij} \right) + C'^{(2)}_{ij} \right) - \mathcal{P}_2 Q_{ij} \right).\label{eq:C''upper}
\end{align}
From $(1+u'_{64})^3 \le 1 + 3u_{64}$ and \eqref{eq:C''upper}, we derive
\begin{equation}\label{eq:errC''}
    \left|D_{64} - C''_{64}\right|
    \le 3u_{64} \cdot \left| D_{64} \right|.
\end{equation}
Thus, from~\eqref{eq:A'B'-D64} and $A'B' = C'_{\mathit{exact}} - \mathcal{P}Q$,
\begin{align}
\left| D_{64} \right|
&= \left| C'^{(1)} + C'^{(2)} - (\mathcal{P}_1 + \mathcal{P}_2)Q \right|\\
&\le \left| (C'^{(1)} + C'^{(2)}) - C'_{\mathit{exact}} \right| 
+ \left| C'_{\mathit{exact}} - \mathcal{P}Q \right|
+ \left| \mathcal{P} - (\mathcal{P}_1 + \mathcal{P}_2) \right| \left| Q \right|\\
&\le 2^{1+\lceil \log_2 \rho \rceil}(N + 2)u_{64}^2\rho\mathcal{P}E + |A'B'|.
\label{eq:upperboundC''exact}
\end{align}
Hence, \eqref{eq:errA'B'}, \eqref{eq:A'B'-D64}, \eqref{eq:errC''}, and \eqref{eq:upperboundC''exact} imply
\begin{align}
    \left| A'B' - C''_{64} \right| 
    &\le 2^{1+\lceil \log_2 \rho \rceil}(N + 2)u_{64}^2\rho\mathcal{P}E
    + 3u_{64}\left( 2^{1+\lceil \log_2 \rho \rceil}(N + 2)u_{64}^2\rho\mathcal{P}E + |A'B'| \right)\\
    &= (1 + 3u_{64})2^{1+\lceil \log_2 \rho \rceil}(N + 2)u_{64}^2\rho\mathcal{P}E + 3u_{64}|A'B'|.
\end{align}

Next, we give the proof for $b = 32$.
From Lemma~\ref{lem:basic-err}, $|\mathcal{P} - \mathcal{P}_1| \le u'_{64}\min(\mathcal{P}, \mathcal{P}_1)$.
Thus, from~\eqref{eq:err_C'_32} and \eqref{eq:upperboundQ},
\begin{align}
    \left| A'B' - D_{32} \right| 
    &= \left| (C'_{\mathit{exact}} - \mathcal{P}Q) - ( C'^{(1)} - \mathcal{P}_1 Q) \right| \\
    &\le \left| C'_{\mathit{exact}} - C'^{(1)} \right| + \left| \mathcal{P} - \mathcal{P}_1 \right| \left|Q\right|\\
    &\le (N+1)u_{64}\rho\mathcal{P}E + u'_{64}\rho\mathcal{P}E
    \le (N+2)u_{64}\rho\mathcal{P}E.\label{eq:A'B'-D32}
\end{align}
From Lemma~\ref{lem:basic-err}, for $1 - u'_{32} \le (\Delta_4)_{ij} \le 1 + u'_{32}$ and $1 - u'_{64} \le (\Delta_5)_{ij} \le 1 + u'_{64}$, 
\begin{equation}
    (C''_{32})_{ij}
    = \mathrm{single}\left(\FL{\FMA{-Q_{ij},\mathcal{P}_1,C'^{(1)}_{ij}}} \right)
    = \mathrm{single}\left((\Delta_5)_{ij}\left( C'^{(1)}_{ij} - \mathcal{P}_1 Q_{ij} \right) \right)
    = (\Delta_4)_{ij}(\Delta_5)_{ij}\left( C'^{(1)}_{ij} - \mathcal{P}_1 Q_{ij} \right).\label{eq:C''upper32}
\end{equation}
Therefore, we obtain
\begin{equation}\label{eq:errC''32}
    \left|D_{32}- C''_{32}\right|
    \le u_{32} \cdot \left| D_{32} \right|.
\end{equation}
From~\eqref{eq:A'B'-D32} and $A'B' = C'_{\mathit{exact}} - \mathcal{P}Q$,
\begin{align}
\left| D_{32} \right|
= \left| C'^{(1)} - \mathcal{P}_1 Q \right|
\le \left| C'^{(1)} - C'_{\mathit{exact}} \right| 
+ \left| C'_{\mathit{exact}} - \mathcal{P}Q \right|
+ \left| \mathcal{P} - \mathcal{P}_1 \right| \left| Q \right|
\le (N+2)u_{64}\rho\mathcal{P}E + |A'B'|.
\label{eq:upperboundC''exact32}
\end{align}
Hence, \eqref{eq:errA'B'}, \eqref{eq:A'B'-D32}, \eqref{eq:errC''32}, and \eqref{eq:upperboundC''exact32} imply
\begin{align}
    \left| A'B' - C''_{32} \right| 
    \le (N+2)u_{64}\rho\mathcal{P}E + u_{32}\left( (N+2)u_{64}\rho\mathcal{P}E + |A'B'| \right)
    = (1 + u_{32})(N+2)u_{64}\rho\mathcal{P}E + u_{32}|A'B'|.
\end{align}

\end{proof}

\subsection{Proof of Theorem~\ref{thm:err}}
\begin{proof}
For $t := 1/\sqrt{2^{5} (\mathcal{P}-1)}$, \eqref{eq:mu_i_upper} and~\eqref{eq:nu_i_upper} in Lemma~\ref{lem1} imply that
\[
    2^{-\mu_i} 2^{-\nu_j}
    \le \frac{1}{2^5(\mathcal{P}-1)} \cdot 2^{\alpha'_i + \beta'_j} 
    = t^2 2^{\alpha'_i + \beta'_j}.
\]
Therefore, from Lemma~\ref{lem4}, for $b \in \{32,64\}$,
\begin{align*}
    \left( \mathrm{diag}(2^{-\mu}) \cdot | C''_{b} - A'B' | \cdot \mathrm{diag}(2^{-\nu}) \right)_{ij}
    \le t^2 2^{\alpha'_i + \beta'_j}R_{b}.
\end{align*}
Thus, from~\eqref{eq:AB_upper} in Lemma~\ref{lem1}, 
we obtain
\begin{align*}
    |(AB - C_{b})_{ij}|
    &\le |(AB - \mathrm{diag}(2^{-\mu}) \cdot C''_{b} \cdot \mathrm{diag}(2^{-\nu}))_{ij}|\\
    &\le |(AB - \mathrm{diag}(2^{-\mu}) \cdot A'B' \cdot \mathrm{diag}(2^{-\nu}))_{ij}|
    + (\mathrm{diag}(2^{-\mu}) \cdot | C''_{b} - A'B' | \cdot \mathrm{diag}(2^{-\nu}))_{ij}\\
    &\le (t |A| v (2^{\beta'})^T
        + t 2^{\alpha'} v^T |B|
        + kt^2 2^{\alpha'} (2^{\beta'})^T)_{ij}
        + (t^2 2^{\alpha'}(2^{\beta'})^T)_{ij}(R_{b})_{ij}\\
    &= (t |A| v (2^{\beta'})^T
        + t 2^{\alpha'} v^T |B|)_{ij}
        + (t^2 2^{\alpha'} (2^{\beta'})^T)_{ij}(kE + R_{b})_{ij}.
\end{align*}

\end{proof}

%====================
\section{Discussion}
\label{sec:Discussion}
%====================
Numerical experiments were conducted on an NVIDIA GeForce RTX 4090 GPU with NVIDIA CUDA Toolkit 13.1.80.
Figures~\ref{fig:comparison_d} and~\ref{fig:comparison_s} compare the theoretical error bounds derived in~\eqref{thm:err-1} and~\eqref{thm:err-2} with the actual numerical errors observed in DGEMM and SGEMM emulation, respectively.
Here, the error is measured as $|AB - C |$, where $C$ denotes the computed result.
Let $m=n=128$ and $k=8192$.
Each element of the test matrices $A \in \mathbb{F}_b^{m \times k}$ and $B \in \mathbb{F}_b^{k \times n}$ was generated as
\[
a_{ij}, b_{ij} \approx (\mathrm{rand}-0.5) \cdot \exp(\mathrm{randn} \cdot \phi),
\]
where $\mathrm{rand} \in (0,1] \subset \mathbb{F}_b$ denotes a uniform random number and $\mathrm{randn} \in \mathbb{F}_b$ denotes a standard normal random number. 
The parameter $\phi$ controls the dynamic range of the input matrices.
In both figures, the observed emulation errors (err\_max and err\_min) are consistently bounded by the corresponding theoretical estimates, confirming the validity of the proposed error analysis.
For reference, the maximum errors of native DGEMM and SGEMM are also shown, illustrating the relative accuracy of the emulation results.
The tighter bounds given by~\eqref{thm:err-1} more accurately capture the range of observed errors, whereas~\eqref{thm:err-2} provides a more conservative estimate.
However, the bound in~\eqref{thm:err-1} involves the quantity $|A'B'|$, whose evaluation has a nontrivial computational cost.
In contrast, in~\eqref{thm:err-2}, $|A'B'|$ is replaced by a constant upper bound $\mathcal{P}/2 \cdot E$, allowing the error bound to be estimated using only matrix-vector products.
This property makes~\eqref{thm:err-2} more practical for inexpensive and robust error estimation, despite its increased conservativeness.

\begin{figure}[htbp]
    \centering
    
    \noindent
    \begin{minipage}[b]{0.325\hsize}\centering
    \includegraphics[width=\hsize]{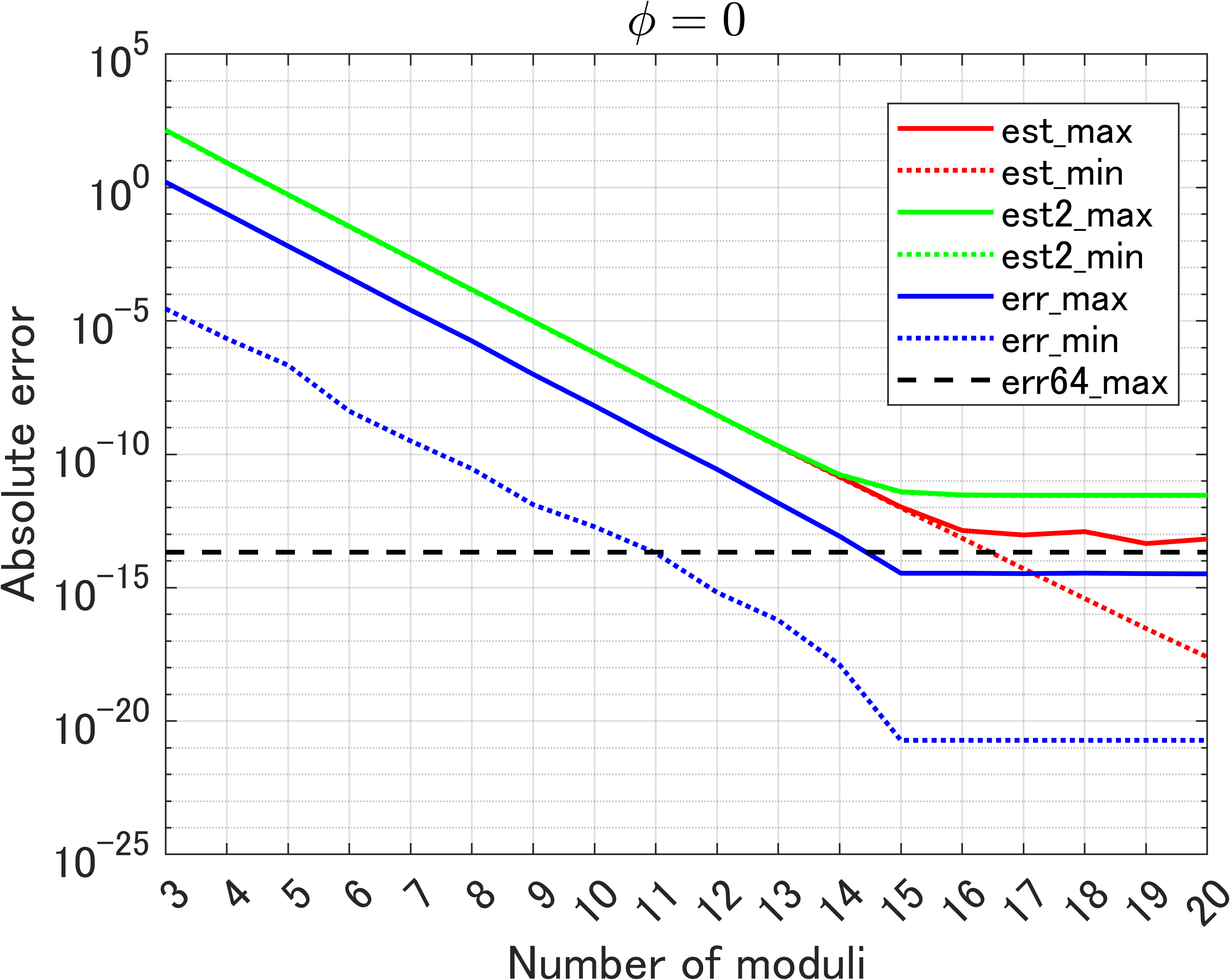}
    \end{minipage}
    \begin{minipage}[b]{0.325\hsize}\centering
    \includegraphics[width=\hsize]{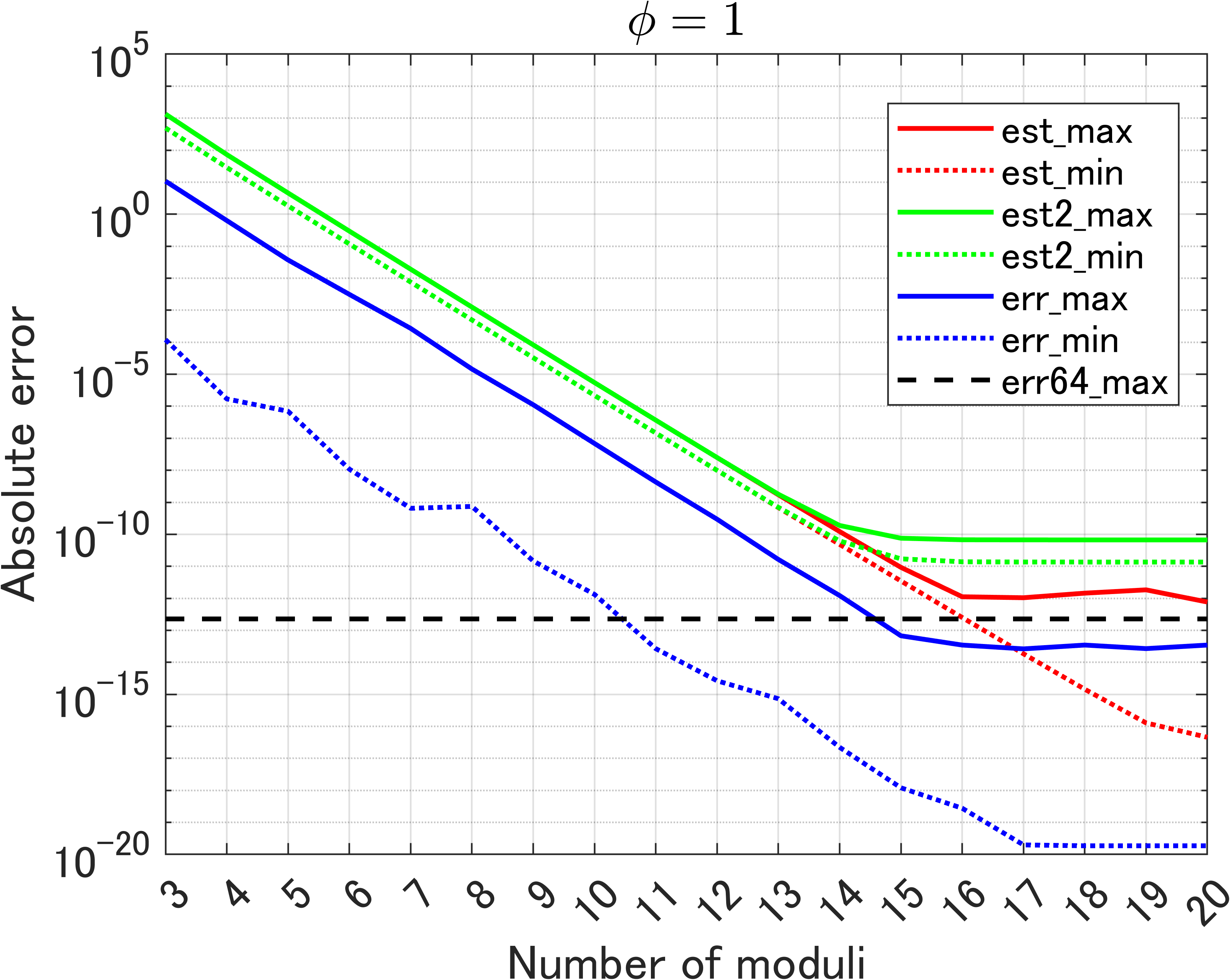}
    \end{minipage}
    \begin{minipage}[b]{0.325\hsize}\centering
    \includegraphics[width=\hsize]{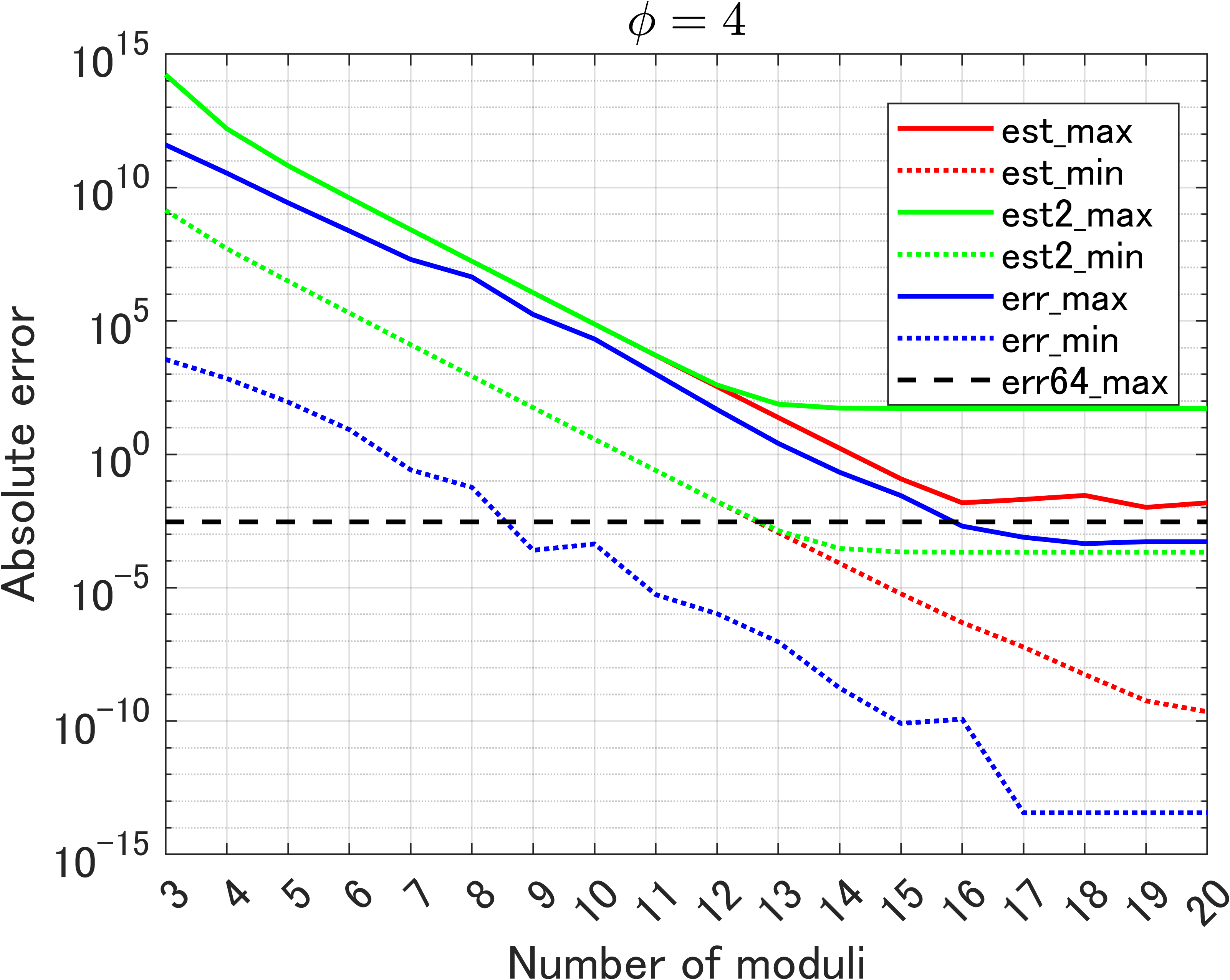}
    \end{minipage}
    
    \caption{Comparison between theoretical error bounds and observed numerical errors for DGEMM emulation. est\_max and est\_min are the maximum and minimum error bounds derived from Theorem~\ref{thm:err} using~\eqref{thm:err-1}, while est2\_max and est2\_min are those obtained from~\eqref{thm:err-2}. err\_max and err\_min are the maximum and minimum values of the actual emulation error, respectively. For reference, err64\_max is the maximum error of native DGEMM.}
    \label{fig:comparison_d}
\end{figure}

\begin{figure}[htbp]
    \centering
    
    \noindent
    \begin{minipage}[b]{0.325\hsize}\centering
    \includegraphics[width=\hsize]{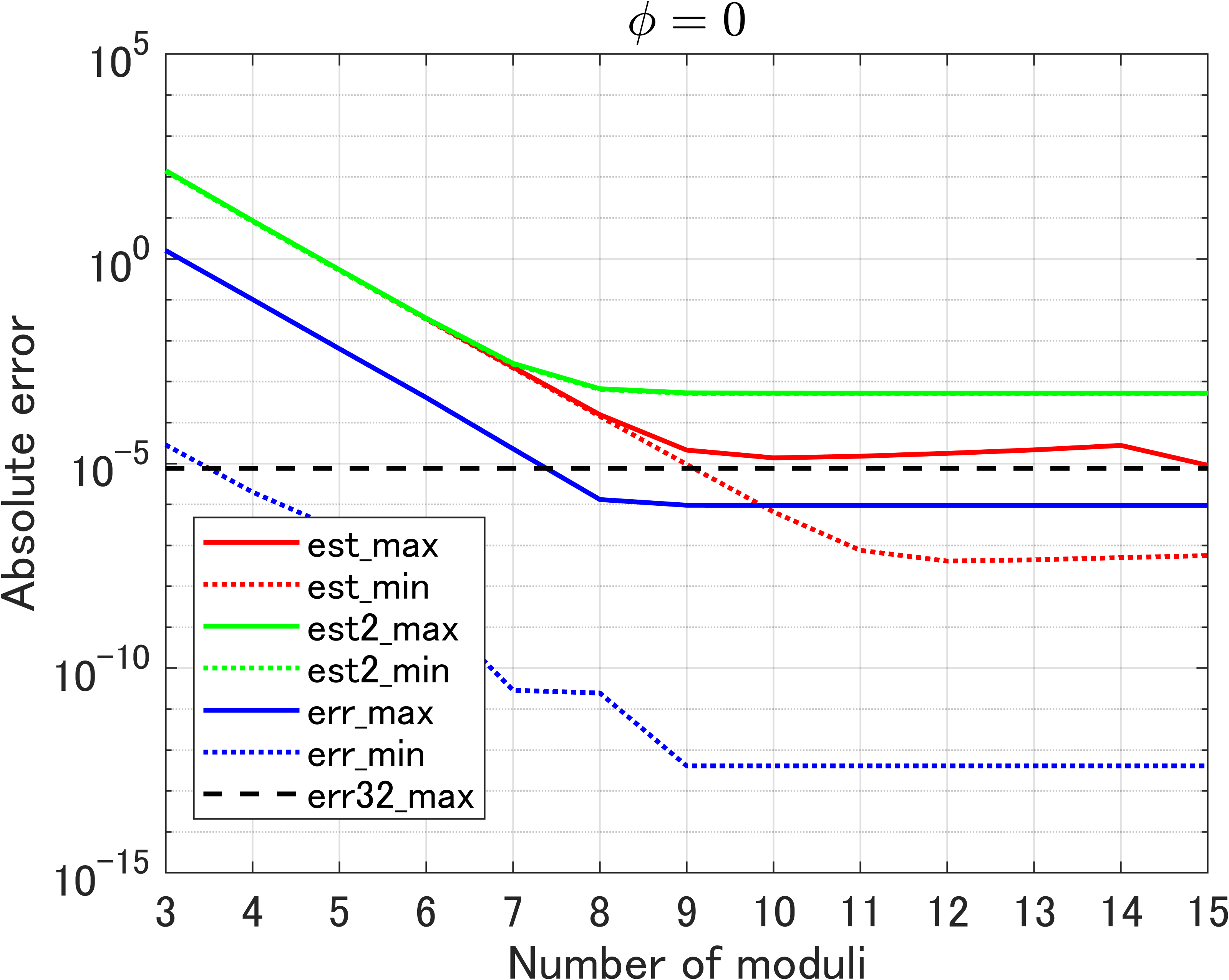}
    \end{minipage}
    \begin{minipage}[b]{0.325\hsize}\centering
    \includegraphics[width=\hsize]{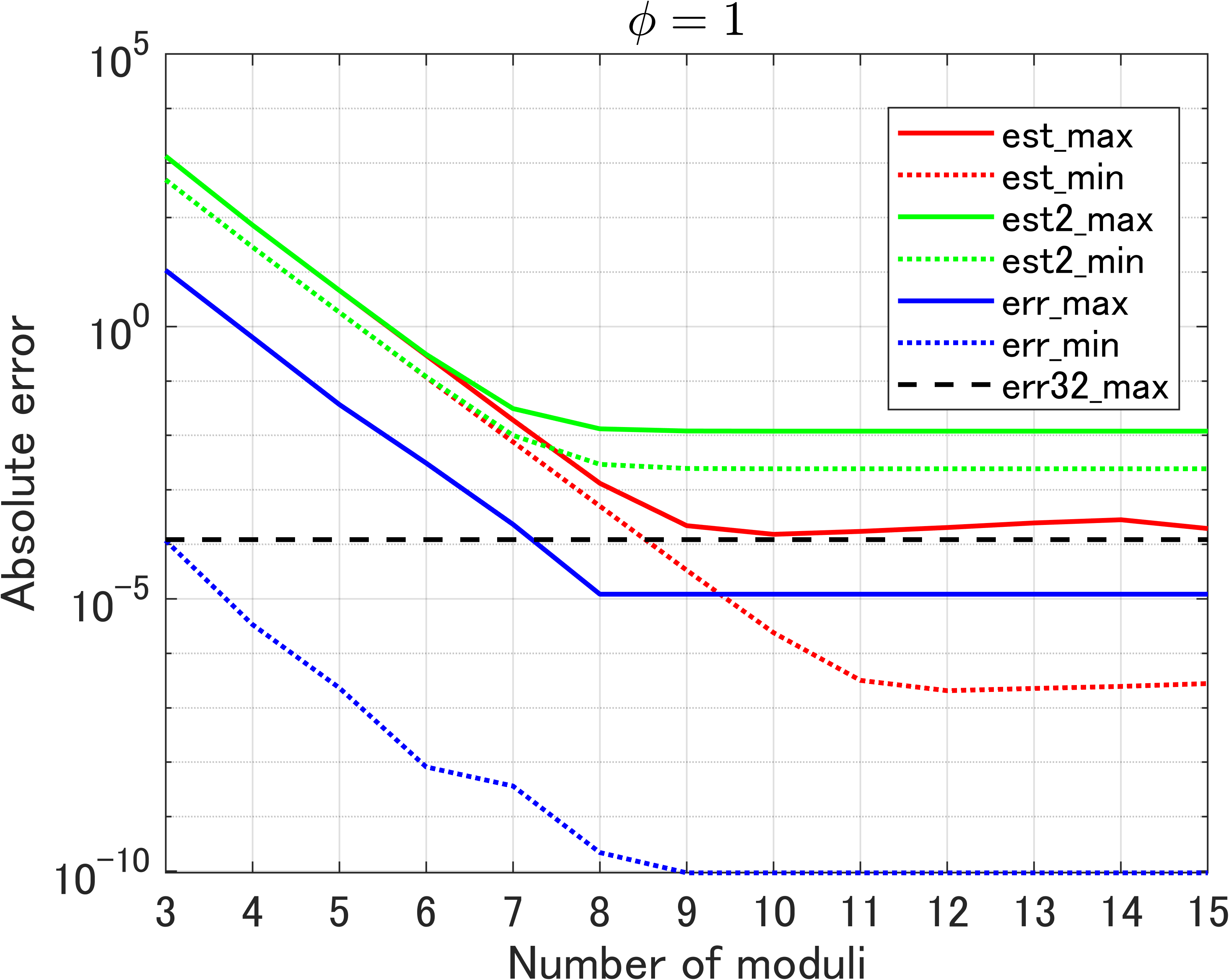}
    \end{minipage}
    \begin{minipage}[b]{0.325\hsize}\centering
    \includegraphics[width=\hsize]{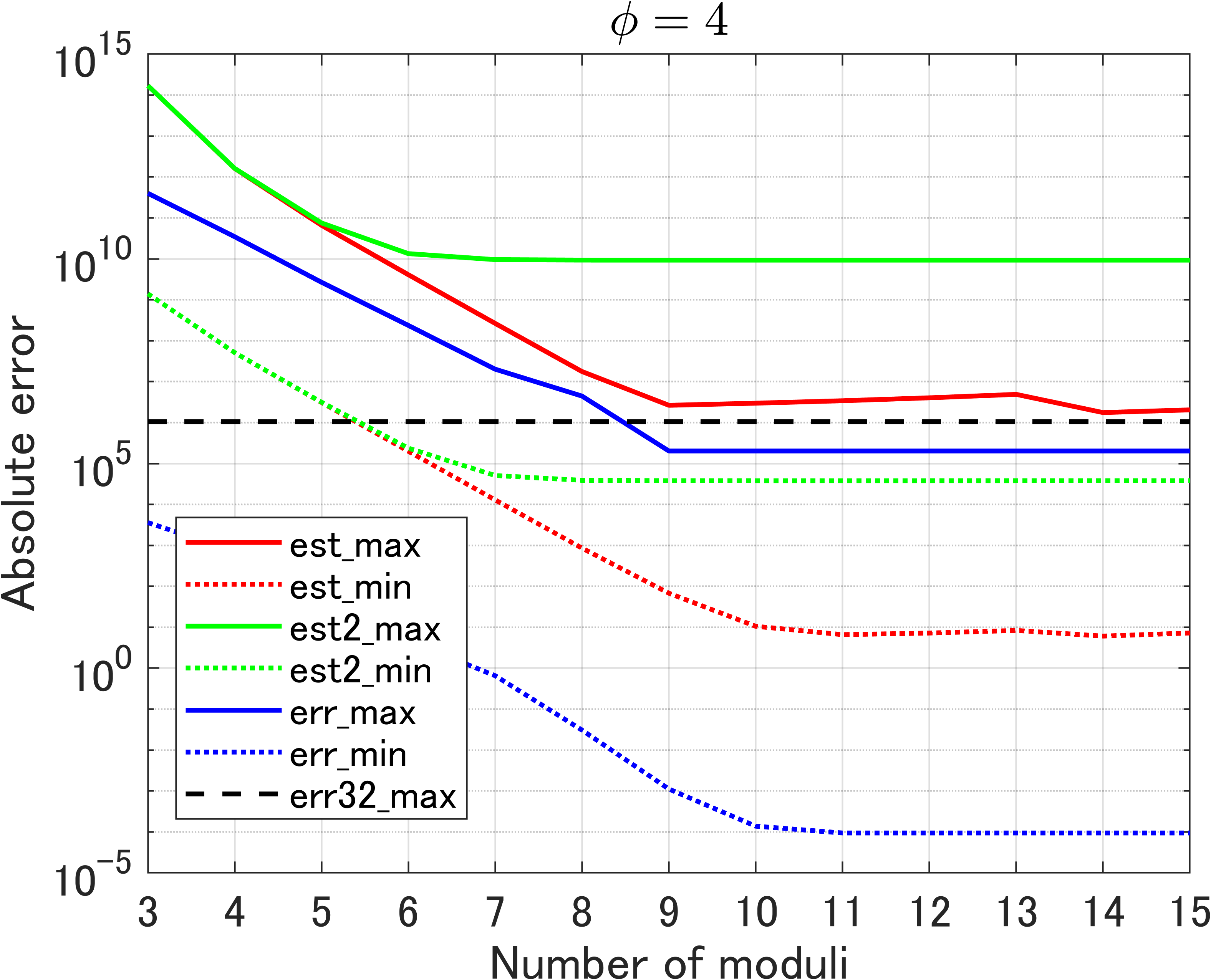}
    \end{minipage}
    
    \caption{Comparison between theoretical error bounds and observed numerical errors for SGEMM emulation. The meanings of the plotted lines are analogous to those in Fig.~\ref{fig:comparison_d}, with err32\_max being the maximum error of native SGEMM.}
    \label{fig:comparison_s}
\end{figure}

%====================
\section{Conclusion}
\label{sec:Conclusion}
%====================
This study provided a deterministic error analysis for SGEMM and DGEMM emulation based on the Ozaki-II scheme.
The proposed analysis theoretically clarifies how the exponent distribution of the input matrices and the number of moduli affect the resulting numerical errors, and explains the accuracy behavior previously observed in numerical experiments.
The present error analysis provides a foundation for developing methods that automatically adjust the number of moduli according to the exponent distribution of the input matrices in order to achieve the desired level of numerical accuracy.
A similar analytical framework can also be applied to the emulation of complex-valued matrix multiplication proposed in~\cite{uchino_ozaki2_complex}, which we leave as future work.

%====================
\section*{Acknowledgements}
%====================
This study was supported by the Japan Society for the Promotion of Science under Grant-in-Aid Numbers 25K03126, 23H03410 (Scientific Research B), and 24K23874 (Research Activity Start-up).

\bibliographystyle{elsarticle-num}
\bibliography{references}

\end{document}